\documentclass{elsarticle}  

\usepackage{graphicx}
\usepackage[cmex10]{amsmath}
\usepackage{amsfonts}
\usepackage{amssymb}
\usepackage{color}
\usepackage{subfigure}

\newenvironment{proof}{\paragraph{Proof}}{\hfill$\square$}

\newtheorem{thm}{Theorem}
\newtheorem{defn}{Definition}
\newtheorem{prop}{Proposition}

\newcommand{\Real}{\mathbb{R}}
\newcommand{\Sphere}{\mathcal{S}}
\newcommand{\Ball}{\mathcal{B}}
\newcommand{\Volume}{\mathcal{V}}
\newcommand{\Disc}{\mathcal{D}}

\begin{document}

\begin{frontmatter}

\title{Approximation with Random Bases: Pro et Contra }

\author[Leic]{Alexander N. Gorban}\ead{ag153@le.ac.uk}
\author[Leic,Leti]{Ivan Yu. Tyukin\corref{cor1}\fnref{f1}}\ead{I.Tyukin@leicester.ac.uk}
\author[Toy]{Danil V. Prokhorov}\ead{dvprokhorov@gmail.com}
\author[Leic,Apical]{Konstantin I. Sofeikov}\ead{sofeykov@gmail.com}

\address[Leic]{University of Leicester, Department of Mathematics, University Road, Leicester, LE1 7RH, UK}
\address[Leti]{Department of Automation and
        Control Processes, St. Petersburg State University of
        Electrical Engineering, Prof. Popova str. 5, Saint-Petersburg, 197376, Russian Federation}
\address[Toy]{Toyota Research Institute NA, Ann Arbor, MI 48105, USA}

\cortext[cor1]{Corresponding author}
\fntext[f1]{The work was supported by Innovate UK Technology Strategy Board (Knowledge Transfer Partnership grant KTP009890) and Russian Foundation for Basic Research (research project No. 15-38-20178).}

\begin{abstract}

In this work we discuss the problem of selecting suitable approximators from families of parameterized elementary functions that are known to be dense in
a Hilbert space of functions. We consider and analyze published procedures, both randomized and deterministic, for selecting elements from these families that have been shown to ensure the rate of convergence in $L_2$ norm of order $O(1/N)$, where $N$ is the number of elements. We show that both {\color{black} randomized and deterministic procedures} are successful {\color{black} if } additional information about the families of functions to be approximated is provided. In {\color{black} the} absence of such additional information one may observe exponential growth of the number of terms needed to approximate the function and/or extreme sensitivity of the outcome of the approximation to parameters. Implications of our analysis for applications of neural networks in modeling and control are illustrated with examples.

\end{abstract}

\end{frontmatter}


\section{Introduction}

The problem of efficient representation and modeling of data is important in many areas of science and engineering. A typical problem in this area involves constructing quantitative models (maps) of the type
\[
\begin{split}
& x_1,x_2,\dots,x_d \mapsto f(x_1,x_2,\dots,x_d),
\end{split}
\]
where $x_1,x_2,\dots,x_d$, $x_i\in\Real$, $i=1,\dots,d$ are variables and $f:\Real^d\rightarrow\Real$ is an unknown functional relation between the variables. The total number, $d$, of variables, determining input data may be large, and physical models of such relations $f(\cdot)$ are not always available.

In {\color{black} the} absence of acceptably detailed prior knowledge of ``true'' models, $f(\cdot)$, a commonly used alternative is to express the function $f(\cdot)$ as a linear combination of known functions, $\varphi_i(\cdot)$, $\varphi:\Real^d\rightarrow\Real$:
\begin{equation}\label{eq:approximation:0}
f(x)\simeq f_N(x)= \sum_{i=1}^N c_i \varphi_i(x), \ c_i\in\Real.
\end{equation}
Numerous classes of functions $\varphi_i(\cdot)$ in (\ref{eq:approximation:0}) have been proposed and analysed to date, starting from  $\sin(\cdot)$, $\cos(\cdot)$ and polynomial functions featured in classical Fourier, Fejer, and Weierstrass results,  wavelets \cite{Trefethen:2013}, \cite{Resnikoff:2002}, and reaching out to linear combinations of sigmoids \cite{Cybenko}
\begin{equation}\label{eq:mlp}
f_N(x)=\sum_{i=1}^N c_i \frac{1}{1+e^{-(w_i^{T}x+b_i)}}, \ w_i\in\Real^d, \ b_i\in\Real,
\end{equation}
radial basis functions \cite{Park93}
\begin{equation}\label{eq:rbf}
f_N(x)=\sum_{i=1}^N c_i e^{(-\|w_i^{T}x+b_i\|^2)} \ w_i\in\Real^d, \ b_i\in\Real,
\end{equation}
and other general functions \cite{Gorban:1998} that are often used in  neural {\color{black} network} literature \cite{Haykin99}. {\color{black} Sometimes} the values of parameters $w_i$, $b_i$ may be pre-selected on the basis of additional prior knowledge, leaving only linear weights $c_i$ for training. If no prior information is available then both nonlinear ($w_i$ and $b_i$) and linear ($c_i$)
parameters, or weights, are typically subject to training on data specific to
the problem at hand (full network training).

One special feature that makes full training of approximators (\ref{eq:mlp}), (\ref{eq:rbf}) particularly attractive is that in addition to their universal approximation capabilities \cite{Cybenko}, \cite{Hornik90}, \cite{Park93} and their homogenous
structure, they are reported to be efficient when the dimension $d$ of input data is relatively high. In particular, if {\it all parameters $w_i$, $c_i$, $b_i$ } are allowed to vary, the
order of convergence rate of the approximation error of a sufficiently smooth function
$f(\cdot)\in\mathcal{C}^{0}[0,1]^d$ as a function of $N$
(the number of elements in the network) is shown to be independent
of the input dimension $d$ \cite{Jones:1992}, \cite{Barron}.
Furthermore, the achievable rate of convergence of the $L_2$-norm of
$f(\cdot)-f_N(\cdot)$ is shown to be $O(1/N^{1/2})$. This contrasts sharply with the rate $O(d^{-1}/N^{1/d})$ corresponding to the worst-case estimate inherent to linear combinations (\ref{eq:approximation:0}) with $\varphi_i(\cdot)$ given or (\ref{eq:mlp}), (\ref{eq:rbf}) with $w_i,b_i$ fixed. In particular, it is shown in \cite{Barron} that if
only linear parameters of (\ref{eq:mlp}), (\ref{eq:rbf}) are
adjusted the approximation error {\it cannot be made} smaller than
$Cd^{-1}/N^{1/d}$, where $C$ is independent on $N$, uniformly for functions satisfying the same smoothness
constraints.

Favorable independence of the order of convergence rates
on the input dimension of the function to be approximated, however, comes at a price. Construction of
such efficient models of data involves a nonlinear optimization
routine searching for the best possible values of $w_i$, $b_i$. The necessity to adjust {\color{black} parameters entering (\ref{eq:mlp}) nonlinearly},
(\ref{eq:rbf}) restricts practical application of these models in {\color{black} a range of
relevant problems} (see e.g., \cite{tpt2002_at}, \cite{tpt2003_tac},
\cite{tpvl2007_tac}, \cite{Annaswamy99} for an overview of potential issues).

An alternative to adjustment of nonlinear parameters of (\ref{eq:mlp}),
(\ref{eq:rbf}) has been proposed in \cite{Schmidt:1992}, \cite{Pao:1994}, \cite{igelpao95} and further developed in \cite{Rahimi08}, \cite{Rahimi08:2} (see also earlier work by F. Rosenblatt \cite{Rosenblatt} in which he
discussed perceptrons with random weights). In these works the nonlinear parameters $w_i$ and $b_i$ are proposed to
be set randomly at the initialization, rather than through training.
The only trainable parameters are those which enter the network
equation linearly, $c_i$.  Random selection of weights in the hidden layers is supposed to generate a set of (basis) functions that is sufficiently rich to approximate a given function by mere linear combinations from this set. This crucially simplifies training in systems featuring such approximators, and renders an otherwise computationally complex nonlinear optimization problem into a much simpler linear one.
Moreover, the rate of  error convergence is argued to be $O(1/N^{1/2})$ \cite{igelpao95}, \cite{Rahimi08}, \cite{Rahimi08:2}.

This brings us to a paradoxical contradiction: on the one hand the reported convergence rate $O(1/N^{1/2})$ which a randomized approximator is supposed to achieve contradicts to the earlier worst-case estimate $O(d^{-1}/N^{1/d})$ obtained in \cite{Barron}. On the other hand numerous studies report successful application of such
randomization ideas (see also comments \cite{lichow97})  in a variety of intelligent control applications
\cite{he05}, \cite{LewisGe06}, \cite{Jag06}, \cite{LewisCampos02}, \cite{Liuetal07} as well as in  machine learning (see e.g. \cite{INS:Wang:2014}, \cite{Neural_Networks:Widrow:2013} and references therein). The question, therefore, is if it is possible to resolve such an apparent controversy? Answering to this question is the main purpose of this contribution.

The paper is organized as follows. We begin the analysis with Section \ref{sec:Analysis} in which {\color{black} we}
review basic reasoning in \cite{igelpao95} and  compare these results
with  \cite{Jones:1992}, \cite{Barron}. We show that, although these
results may seem inconsistent, they have been derived for
different performance criteria. The worst-case estimate in \cite{Barron} is ``crisp'', whereas the convergence rate in
\cite{igelpao95} is probabilistic. That is it involves a measure function with respect to which the rate $O(1/N^{1/2})$ is assured. Introduction of measure into the problem brings out a range of interesting consequences that is discussed and analyzed in Section \ref{sec:main}. There we  approach data approximation problem as that of representation of a given vector by {\color{black} a} linear {\color{black} combination} of randomly chosen {\color{black} vectors} in high dimensions. A simple logic suggests that if one takes $m\leq n$ random vectors in $\Real^n$ then {\color{black} it can be expected} that with probability one these vectors will be linearly independent because the set of linearly dependent corteges $\{x_1,\ldots , x_m\}$ ($m\leq n$) is a proper algebraic subset in the space of corteges $(\Real^n)^m$. We can select $n$ random vectors, {\color{black} and} with probability one it will be a basis. Every given data vector $y$ can be represented by coordinates in this basis. If these $n$ vectors are (accidentally) too close to dependence we can generate few more vectors that will enable us to represent the data vector $y$. We will show, however, that this simple and correct reasoning looses its credibility in high dimensions.  We show that in an $n$-dimensional {\color{black} unit} cube $[-1,1]^n$ a randomly generated vector $x$ will be almost orthogonal to a given data vector $y$ (the angle between $x$ and $y$ will be close to $\pi/2$ with probability close to one). To compensate for this {\it waist concentration} effect \cite{GAFA:Gromov:2003}, \cite{Gromov:1999}, \cite{Gorban:2006} one needs to generate exponentially many random vectors. Typicality of such exponential growth is an inherent feature of  high-dimensional data representation, including problems of data approximation and modelling.  Moreover, for high-dimensional data representation the following two seemingly contradictory situations are typical in some sense:

\begin{itemize}
 \item {\it with probability close to one linear combinations of $n-k$ random vectors approximate any normalized vector with accuracy  $\varepsilon$ if $k \ll n$ and no constraints on the values of coefficients in linear combinations are imposed (and this probability is one, if $k=0$);}

\item {\it with probability close to one an exponentially large number $N$ of random vectors are pairwise almost orthogonal and do not span an arbitrarily selected normalized vector if coefficients in linear combinations are not allowed to be arbitrarily large.}

\end{itemize}

{\color{black} The approach to assess effective dimensionality of spaces based on  $\epsilon$-{\it quasi\-or\-tho\-go\-na\-lity} was introduced in \cite{Kurkova}, \cite{Hecht}. The authors demonstrated that in high dimensions there {\it exist} exponentially large sets of quasiorthogonal vectors. This observation has been exploited in the random indexing literature \cite{Salhgren2005}. Here we show that not only such sets exist, but also that they are typical in some sense}. Implications of our analysis are illustrated in Section \ref{sec:discussion} followed by few examples
presented in Section \ref{sec:Example}.  Section \ref{sec:Conclusion} concludes the
paper.

\section{Notation}\label{sec:notation}

Throughout the paper the following notational agreements are used

\begin{itemize}
\item $\Real$  denotes the field of real number;
\item $\Real^n$ stands for the $n$-dimensional linear space over the field of reals;
\item let $x\in\Real^n$, then $\|x\|$ is the Euclidean norm of $x$: $\|x\|=\sqrt{x_1^2+\cdots + x_n^2}$;
\item if $x,y$ are two non-zero vectors from $\Real^n$ then $\angle(x,y)$ denotes an angle between these vectors;
\item symbol $|\cdot |_{\Real^n}$ is reserved to denote an arbitrary norm in $\Real^n$;
\item $\Sphere^{n-1}(R)$ denotes an $n-1$-sphere of radius $R$ centred at $0$: $\Sphere^{n-1}(R)=\{x\in\Real^n | \ \|x\|=R \ \}$
\item $\mu$ is the normalized Lebesgue measure on $\Sphere^{m-1}(1)$: $\mu(\Sphere^{m-1}(1))=1$.
\item $\Ball^{n}(R)$ denotes a $n$-ball of radius $R$ centered at $0$: $\Ball^n(R)=\{x\in\Real^n | \ \|x\|\leq R\}$
\item $\Volume(\Xi)$ is the Lebesgue volume of  $\Xi \subset \Real^n$.
\item $\Disc^{n-1}(R)$ stands for a $n-1$-disc in the $n$-ball $\Ball^n(R)$ corresponding to its largest equator, and $\Disc^{n-1}_\delta(R)$ is its $\delta$-thickening.
\item Let $f:[0,1]^d\rightarrow\Real$ be a continuous function, then
\[
\|f\|^2=\langle f,f \rangle=\int_{[0,1]^d} f(x)f(x) d x,
\]
denotes the $L_2$-norm of $f$.
\end{itemize}

\section{Preliminary analysis and motivation}\label{sec:Analysis}

In this section we review and compare the
so-called {\it greedy approximation} upon which the famous Barron's
construction is based \cite{Barron} and the Random Vector Functional-Link (RVFL) network
\cite{igelpao95} in which the basis functions are randomly chosen,
and only their linear parameters are optimized. As we show below, despite the error convergence rates appear to be
of the same order, they are in essence dramatically different. One rate is crisp in the sense that it is a worst-case estimate for all functions from a given class. The other one is probabilistic in nature, and hence holds in probability. The consequences of the latter are further investigated in Section \ref{sec:main}.


Consider the following class of problems. Suppose that
$g:\Real\rightarrow\Real$ be a function such that
\[
\|g\|\leq M, \ M\in\Real_{>0},
\]
and
\[
\mathcal{G}=\{g(w^{T}x + b)\}, \ w\in\Real^d, \ b\in\Real,
\]
be a family of parameterized functions $g(\cdot)$. Let $f\in\mathcal{C}^{0}([0,1]^n)$, and let $f$ belong to a convex hull of $\mathcal{G}$.
In other words there is a sequence of $w_i$, $b_i$, and $c_i$ such that
\[
f(x)=\sum_{i=1}^{\infty} c_i g(w_i^{T}x + b_i), \ \sum_{i=1}^{\infty}c_i=1, \ c_i\geq 0.
\]
{\color{black} We are interested in approximating $f$ by the finite sum}
\begin{equation}\label{eq:approximation}
f_N(x)=\sum_{i=1}^N c_i g(w_i^{T}x+b_i).
\end{equation}
It is important to know how approximation errors, defined in some meaningful sense, may decay  with $N$, and what the computational costs for achieving this rate of decay {\color{black} are}?

\subsection{Greedy approximation and Jones Lemma}

In order to answer the question above one needs first to determine
the error of approximation. It is natural for functions from $L_2$
to define the approximation error as follows:
\begin{equation}\label{eq:error_greedy}
e_N=\|f_N - f\|
\end{equation}

The classical Jones iteration \cite{Jones:1992} (refined later by
Barron \cite{Barron}) provides us with the following estimate of
achievable convergence rate:
\begin{equation}\label{eq:rate:greedy}
\begin{split}
e_N^2&\leq \frac{M'^2 e_0^2}{N e_0^2 + M'^2}, \   M'> \sup_{g\in \mathcal{G}} \|g\|+\|f\|.
\end{split}
\end{equation}
The rate of convergence depends on $d$ only through the $L_2$-norms
of $f_0$, $g$, and $f$. The iteration itself is deterministic and
can be described as follows:
\begin{equation}\label{eq:Jones_iteration}
\begin{split}
f_{N+1}(x)&=(1-\alpha_N)f_N(x) + \alpha_N g(w_N^{T}x+b_N)\\
\alpha_N&= \frac{{\color{black}e_N}^2}{M''^2 + {\color{black}e_N}^2}, \ \ M''>M'
\end{split}
\end{equation}
where parameters ${\color{black}w_N},b_N$ of $g$ {\color{black}are} chosen such that the following condition holds
\begin{equation}\label{eq:Jones_iteration:g_n}
\langle f_N-f,g(w_N^T \cdot + b_N)-f\rangle < \frac{((M'')^2-(M')^2) e_N^2}{2(M'')^2},
\end{equation}
This choice is always possible (see \cite{Jones:1992} for
details) as long as the function $f$ is in the convex hull of $\mathcal{G}$.

According to (\ref{eq:rate:greedy}) the rate of convergence of  such
approximators is estimated as
\[
e_N^2 =  O(1/N).
\]
This convergence estimate is {\it guaranteed} because it is the
upper bound for the approximation error at the {\color{black}$N$}th step of
iteration (\ref{eq:Jones_iteration}).


\subsection{Approximation by linear combinations of functions with randomly chosen parameters}

We now turn our attention to the result in \cite{igelpao95}. In this
approximator the original function $f(\cdot)$ is assumed to have the
following integral representation\footnote{We keep the original
notation of \cite{igelpao95} which uses both $\omega$ and $w$ for
the sake of consistency.}
\begin{equation}\label{eq:approximation:2}
f(x)=\lim_{\alpha\rightarrow\infty}\lim_{\Omega\rightarrow\infty} \int_{W^d} F_{\alpha,\Omega}(\omega)g(\alpha w^{T}x+b)d \omega,
\end{equation}
where $g:\Real\rightarrow\Real$ is a non-trivial function from $L_2$:
\begin{equation}\label{eq:square_int}
0<\int_{\Real}g^2(s)ds < \infty,
\end{equation}
where
$\omega=(y,w,u)\in\Real^{d}\times\Real^d\times[-\Omega,\Omega]$, $\Omega\in\Real_{>0}$, $W^d=[-2\Omega;2\Omega]\times I^d \times V^d$, $V^d=[0;\Omega]\times[-\Omega;\Omega]^{d-1}$, $b=-(\alpha w^{T} y + u)$ and
\[
F_{\alpha,\Omega}(\omega)\sim \frac{\alpha\prod_{i=1}^d w_i}{\Omega^d 2^{d-1}} f(y) \cos_{\Omega}(u), \ \cos_\Omega(u)=\left\{\begin{array}{l}
                                                                                                                               \cos(u), \ u\in[-\Omega,\Omega]\\
                                                                                                                                0, \ u\notin [-\Omega,\Omega]
                                                                                                                                \end{array}\right.
\]
(see \cite{igelpao95}, \cite{lichow97} for more detailed description).
Function $g(\cdot)$ induces a parameterized basis. Indeed if we were
to take integral (\ref{eq:approximation:2}) in quadratures for
sufficiently large values of $\alpha$ and $\Omega$, we would then
express $f(x)$ by the following sums of parameterized $g(\alpha
w^{T}x+b)$ \cite{igelpao95}:
\begin{equation}\label{eq:approximation:R}
f(x)\simeq \sum c_{i} g(\alpha w_i^{T}x + b_i), \ b_i=-
\alpha(w_i^{T} y_i + u_i)
\end{equation}
The summation in (\ref{eq:approximation:R}) is taken over points
$\omega_i$ in $W^d$, and $c_i$ are weighting coefficients.
Variables $\alpha$ in (\ref{eq:approximation:2}) and $\alpha_N$ in
(\ref{eq:Jones_iteration}) play different roles in each
{\color{black} approximation scheme}. In (\ref{eq:Jones_iteration}) the value of
{\color{black}$\alpha_N$} is set to ensure that the approximation error is
decreasing with every iteration, and  in (\ref{eq:approximation:2})
it stands for a scaling factor of random sampling.

The main idea of \cite{igelpao95} is to approximate integral
representation (\ref{eq:approximation:2}) of $f(x)$ using the Monte-Carlo integration method as
\begin{equation}\label{eq:MC_approximation}
\begin{split}
f(x)&\simeq \frac{4 \Omega^d}{N} \lim_{\alpha\rightarrow\infty}\lim_{\Omega\rightarrow\infty} \sum_{k=1}^{\color{black}N} F_{\alpha,\Omega}(\omega_k)g(\alpha w^{T}_k x+b_k)\\
&=\lim_{\alpha\rightarrow\infty}\lim_{\Omega\rightarrow\infty} \sum_{k=1}^N c_{k,\Omega}(\alpha,\omega_k)g(\alpha w^{T}_kx+b_k)\\
&=f_{N,\omega,\Omega}(x),
\end{split}
\end{equation}
where the coefficients $c_k(\alpha,\omega_k)$ are defined as
\begin{equation}\label{eq:MC_coefficients}
c_{k,\Omega}(\alpha,\omega_k)= \frac{1}{N}\frac{4\alpha\prod_{i=1}^d w_{k,i}}{2^{d-1}} \cos_{\Omega}(u_k) f(y_k)
\end{equation}
and $\omega_k=(y_k,w_k,u_k)$ are randomly sampled in $W^d$ (domain
of parameters, i.e., weights and biases of the network).

When the number of samples, $N$, i.e., {\em the network size}, is
large, then
\begin{equation}\label{eq:error_random:1}
E_\omega(N)=\sqrt{E_\omega \int_K |f(x)-f_{N,\omega,\Omega}(x)|^2 dx }, \ K\subset[0,1]^d,
\end{equation}
converges to zero asymptotically for large $N$. In particular:
\begin{thm}[Igel'nik and Pao \cite{igelpao95}]\label{theorem:igelnik} For any compact $K$, $K\subset [0,1]^d$, $K\neq [0,1]^d$ and any absolutely integrable function $g$  satisfying (\ref{eq:square_int}) there exists a sequence of $f_{N,\omega,\Omega}$ and a sequence probability measures $\{\mu_{N,\Omega,\alpha}\}$ such that
\[
\lim_{N\rightarrow {\color{black}\infty}} E_\omega(N) =0.
\]
\end{thm}
Furthermore, under some additional restrictions on the functions $g,f$ the expectation $E_\omega(N)$ is shown to obey (Theorem 3 in \cite{igelpao95}):
\begin{equation}\label{eq:error_random:1.5}
E_\omega(N)\leq \frac{C}{N^{1/2}}.
\end{equation}

Similar results have been reported in \cite{Rahimi08:2}. There the authors look at the following classes of functions:
\[
\mathcal{F}_p=\{f(x)=\int_{\Omega} \alpha(\omega) \phi(x,\omega) d\omega  \ | \ |\alpha(\omega)|\leq C p(\omega)\}, \ \sup_{x,w} |\phi(x,w)|\leq 1,
\]
where $p$ is a distribution on $\Omega$, and
\[
\mathcal{F}_\omega=\{f(x)=\sum_{k=1}^N \alpha_k \phi(x,\omega_k) \ | |\alpha_k|\leq C/N \}.
\]
For the chosen classes of functions they establish the following
\begin{thm}[cf. Lemma 1 in \cite{Rahimi08:2}]\label{theorem:Rahimi} Let $f$ be a function from $\mathcal{F}_p$ and $w_1$,$\dots$, $w_N$ be drawn iid from $p$. Then for any $1>\delta>0$ with probability at least $1-\delta$ over $w_1,\dots,w_N$ there exists a function $f_N$ from $\mathcal{F}_\omega$ so that
\begin{equation}\label{eq:error_random:2}
\|f-f_N\|\leq \frac{C}{\sqrt{N}}\left(1+2\sqrt{\log\frac{1}{\delta}}\right).
\end{equation}
\end{thm}

%
At the first glance errors (\ref{eq:error_greedy}), (\ref{eq:error_random:1}), (\ref{eq:error_random:2}) and their respective convergence rates look very similar. Yet, they are fundamentally different in that (\ref{eq:error_greedy}) is deterministic and its convergence rate is crisp, whereas the other two estimates have an additional element - a probability measure - characterizing their convergence rate. Introduction of this measure enables to ``ignore'' worst-case elements corresponding to the rate of convergence  $O(d^{-1}/N^{1/d})$. At the same time, it imposes some limitations too since there always is a non-zero probability of an unlucky draw from the probability distribution which will require re-initialization at some stage {\color{black} of approximation}.
 Furthermore, explicit presence of $1/\delta$ in (\ref{eq:error_random:2})\footnote{The very same term is implicit in the error rate derived for Monte-Carlo inspired approach in \cite{igelpao95}. This follows immediately from the Chebyshev inequality. Indeed, if $X_1,\dots,X_N$ are iid random variables with the same mean, $\mu$ and variance $\sigma^2$ then the probability that $|(X_1+\cdots + X_N)/N -\mu| \geq \delta$ is assured to be at most $\sigma^2/(\delta^2 N)$.} suggest that this rapid convergence of order $1/N^{1/2}$ is assured only up to a given and fixed tolerance. These features of randomized approximators should thus be considered with special care in applications.

There is one additional point that is inherent to all randomized approximators and yet it is rarely addressed in practice. This is a measure concentration effect that we will present and consider in detail in the next section.

\section{Randomized Data Approximation and Measure Concentration}\label{sec:main}

So far we considered the problem of data approximation and modelling from the function approximation point of view. Let us, however, look at this problem from a slightly different perspective (cf. \cite{Neural_Networks:Widrow:2013}). Suppose that the data points, $x$, are labeled by real numbers $t_i$ so that each $i$-th data point is uniquely determined by its label $t_i$. This means that the problem can be rephrased as that of representing a combined vector $y=(f(x(t_1)),f(x(t_2)),\dots,f(x(t_n)))$ by linear combinations of  $h_i=(g(w_i^{T}x(t_1)+b_i),g(w_i^T x(t_2)+b_i),\dots,g(w^T_i x(t_n)+b_i))$
so that the error
\[
y-\sum_{i=1}^{m} c_i h_i
\]
is minimized. If our choice of vectors $h_1,\dots,h_m$ is good enough to represent an arbitrary vector $y$ with acceptable precision then we can be assured that the above approximation model is viable. The questions, however, are
\begin{itemize}
\item[1) ] if the choice of $h_i$ is randomized then how many vectors one must choose to ensure desired universality and accuracy?
\item[2) ] how robust such representations are with respect to small perturbation of data?
\end{itemize}
These questions are addressed in the next sections. For simplicity we ignore that the vectors $h_i$ are in fact generated by functions $g(\cdot)$ and proceed assuming that they are simply sampled randomly in a hypercube.

\subsection{Linear dependence and independence and related elementary properties}\label{subsec:linear_dependence}

Let us first recall standard notions of linear dependence and independence.

\begin{defn}[Linear Dependence] A system of vectors $h_1, h_2, \dots, h_m$, $h_i\in\Real^n$, $i=1,\dots,m$  is said to be linearly dependent if there exist $c_1, c_2, \dots, c_m$, $c_i\in\Real$, $i=1,\dots,m$ such that
\[
h_1 c_1 + h_2 c_2 + \cdots +c_m h_m =0
\]
and  at least one of $c_1,c_2,\dots,c_m$ is not equal to zero.
\end{defn}
The same definition can be rewritten in the vector-matrix notation as
\begin{equation}\label{eq:linear_dependence}
\exists \ c\in\Real^m, \ c\neq 0: H c =0,
\end{equation}
where
\[
H=\left(h_1 \ h_2  \cdots \  h_m\right)
\]
is an $n\times m$ matrix formed by $h_1,\dots,h_m$.

\begin{defn}[Linear Independence]\label{defn:linear_independence} A system of vectors $h_1, h_2, \dots, h_m$, $h_i\in\Real^n$, $i=1,\dots,m$ is said to be linearly dependent if it is not linearly independent. In other words,
\begin{equation}\label{eq:linear_independence}
H c \neq 0 \ \forall \ c\in\Real^m: \ c\neq 0.
\end{equation}
\end{defn}

A very simple fact follows immediately from Definition \ref{defn:linear_independence}

\begin{prop}\label{prop:linear_indepencence} Consider a system of vectors $h_1,\dots,h_m$, and let $H=(h_1$ $h_2$  $\cdots$ $h_m)$. Let $\left|\cdot\right|_{\Real^n}$ be a norm on $\Real^n$, then the system $h_1,\dots,h_m$ is linearly independent iff there exists an $\varepsilon>0$ such that
\begin{equation}\label{eq:linear_independence_form}
\left|H x\right|_{\Real^n}>\varepsilon \ \forall \ x\in \mathcal{S}^{m-1}(1).
\end{equation}
\end{prop}
\begin{proof} Suppose first that the system $h_1,\dots,h_m$ is linearly  independent. Hence (\ref{eq:linear_independence}) holds, and given that $x\neq 0$ for all $x\in\mathcal{S}^{m-1}(1)$ we thus obtain that $\left|H x\right|_{\Real^n}>0$ for all $x\in\mathcal{S}^{m-1}(1)$. Noticing that $\left|\cdot\right|_{\Real^n}$ is continuous and $\mathcal{S}^{m-1}(1)$ is compact we conclude that $\left|\cdot\right|_{\Real^n}$  takes minimal and maximal values on $\mathcal{S}^{m-1}(1)$. Let
\[
\varepsilon=\min_{x\in\mathcal{S}^{m-1}(1)} \left|Hx\right|_{\Real^n}.
\]
The minimum of $\left|\cdot\right|_{\Real^n}$ on $\mathcal{S}^{m-1}(1)$ is separated away from $0$ since otherwise there {\color{black} exists} a $x\in\mathcal{S}^{m-1}(1)$ such that $Hx=0$. Thus (\ref{eq:linear_independence_form}) holds.

Let us now show that (\ref{eq:linear_independence_form}) implies (\ref{eq:linear_independence}). Indeed, for any $c\in\Real^m$, $c\neq 0$ there is an $x\in\mathcal{S}^{m-1}(1)$ and an $\alpha\in\Real$, $\alpha\neq 0$ such that $c=\alpha x$.  Hence $\left|Hc\right|_{\Real^n}=|\alpha|\left|Hx\right|_{\Real^n}>|\alpha|\varepsilon > 0$, which automatically assures that (\ref{eq:linear_independence}) holds.
\end{proof}

\subsection{Quantification of linear independence}\label{subsec:quantification}

Standard notions of linear dependence and independence are, unfortunately, not always easy to assess numerically when the values of $\varepsilon$ in (\ref{eq:linear_independence_form}) are small. {\color{black} Further}, checking that for a given system of vectors $h_1,\dots,h_m$, some $\varepsilon$, and all $x\in\Sphere^{m-1}(1)$ the following holds
\begin{equation}\label{eq:linear_independence_relation}
\left|Hx\right|_{\Real^n}>\varepsilon
\end{equation}
may not always be {\color{black} feasible or desirable.}

Two ways to relax and quantify the conventional notion of linear independence are obviated in Proposition \ref{prop:linear_indepencence}. These are 1) the  value of $\varepsilon$ in (\ref{eq:linear_independence_form}), and 2) a possibility of introducing a finite measure on $\Sphere^{m-1}(1)$  that determines a proportion of $x$ from $\Sphere^{m-1}(1)$ which satisfy (\ref{eq:linear_independence_relation}).

\begin{defn}[Almost Linear Independence]\label{defn:almost_linear_independence} Let $h_1,\dots,h_m$ be a system of $m$ normalized vectors from $\Real^n$: $\left|h_i\right|_{\Real^n}=1$, $i=1,\dots,m$.  We will say that the system is $(\varepsilon,\theta)$-linearly independent (almost linearly independent) with respect to $\mu$ if
\begin{equation}\label{eq:almost_linear_independence}
\mu(\{ x \in\Sphere^{m-1}(1) \ | \ \left|H x\right|_{\Real^n}\geq\varepsilon\})\geq 1-\theta.
\end{equation}
\end{defn}

Similarly, one can define an alternative quantification of linear {\color{black} dependence}:

\begin{defn}[Almost Linear Dependence]\label{defn:almost_linear_dependence} Let $h_1,\dots,h_m$ be a system of $m$ normalized vectors from $\Real^n$: $\left|h_i\right|_{\Real^n}=1$, $i=1,\dots,m$. We will say that the system is $(\varepsilon,\theta)$-linearly dependent (almost linearly {\color{black}dependent}) with respect to $\mu$ if
\[
\mu(\{ x \in\Sphere^{m-1}(1) \ | \ \left|H x\right|_{\Real^n}\leq \varepsilon\})\geq 1-\theta.
\]
\end{defn}

Notice that the notions of almost linear independence and almost linear dependence introduced in Definitions \ref{defn:almost_linear_independence}, \ref{defn:almost_linear_dependence} are consistent with conventional notions in the sense that the latter can be viewed as limiting cases of the former. Indeed, if $\mu$ is a surface area then setting $\theta=0$ in Definition \ref{defn:almost_linear_independence} and picking $\varepsilon$ small enough one obtains an equivalent of Definition \ref{defn:linear_independence}.

The above measure, or ``probabilistic'', quantification of linear independence has significant implications for data representation in applications. As we shall see in the next sections two seemingly exclusive extremes are likely to hold in higher dimensions. First, almost all points of an $n$-ball concentrate in an $\epsilon$-thickening of an $n-1$-disc. This means that for $m$ sufficiently large a family of randomly and independently chosen vectors $h_1,h_2,\dots,h_m$ becomes almost linearly dependent. This phenomenon is well-known in the literature (see e.g. \cite{Gromov:1999}, \cite{GAFA:Gromov:2003}, \cite{Gorban:2006}
 as well as classical works of Maxwell, Levy, and Gibbs; in Data Mining applications measure concentration effects have been discussed in \cite{Pestov}).  Yet, the values of $m$ for which almost linear independence persists may be exponentially large. Furthermore, the latter situation holds with a probability that is close to one. This second extreme can be rephrased as that the number of almost orthogonal vectors grows exponentially with dimension. More formal statements and justifications are provided in Propositions \ref{prop:balls_concentration}, \ref{prop:disc_concentration} below.

\subsection{Measure concentration}\label{subsec:measure_concentration}

We begin with the following statement

\begin{prop}\label{prop:balls_concentration} {\color{black} Let $\Ball^n(R)$ be an $n$-ball of radius $R$ in $\Real^n$, and $0<\delta<R$. Then} the following estimate holds:
\begin{equation}\label{eq:balls_concentration}
 \frac{\Volume(\Ball^n(R))-{\color{black}\Volume(}\Ball^n(R-\delta))}{{\color{black}\Volume}(\Ball^n(R))} {\color{black}>} 1-e^{-\frac{n \delta}{R}}.
\end{equation}
\end{prop}
\begin{proof} Noticing that $\Volume(\Ball^n(R))=C_n R^{n}$, where $C_n$ is a constant independent on $R$ we conclude that
\[
\begin{split}
&\frac{\Volume(\Ball^n(R))-\Volume(\Ball^n(R-\delta))}{\Volume(\Ball^n(R))}=1-\frac{(R-\delta)^n}{R^n}=1- \left(1-\frac{\delta}{R}\right)^n.
\end{split}
\]
{\color{black} Next we recall that the following elementary inequality holds for all $0<x<1$:
\begin{equation}\label{eq:exp_inequality}
(1-x)\frac{1}{e}<(1-x)^{\frac{1}{x}}<\frac{1}{e}.
\end{equation}
Indeed, with respect to the right part of (\ref{eq:exp_inequality}), $(1-x)^{\frac{1}{x}}<{e^{-1}}$, we notice that
$(1-x)^{\frac{1}{x}} < e^{-1} \Leftrightarrow 1-x < e^{-x}$. The function $y=e^{-x}$ is strictly convex on $\Real$, and $y=1-x$ is its first-order Taylor approximation at $x=0$. Thus that $(1-x)^{\frac{1}{x}}<\frac{1}{e}$ holds true for $0<x<1$ is the consequence of the strict convexity of the exponential $e^{-x}$.

In order to see that the left part of  (\ref{eq:exp_inequality}) holds true too, consider the following chain of equivalent inequalities for $0<x<1$:
\[
(1-x) e^{-1}< (1-x)^{\frac{1}{x}} \Leftrightarrow  e^{-x}<(1-x)^{1-x}  \Leftrightarrow -x < (1-x)\ln (1-x).
\]
Again, $y=(1-x)\ln (1-x)$ is strictly convex on $(-\infty,1)$, and $y=-x$ is its first Taylor approximation at $x=0$. Hence the left part of the inequality, $(1-x) e^{-1}< (1-x)^{\frac{1}{x}}$, is the consequence of strict convexity of $(1-x)\ln (1-x)$.

Using (\ref{eq:exp_inequality}) one can derive that
\[
\left[e\left(1-\frac{\delta}{R}\right)^{-1}\right]^{- \frac{n\delta}{R}}<\left(1-\frac{\delta}{R}\right)^n < e^{- \frac{n\delta}{R}}.
\]
Moreover,  for $\delta/R$ sufficiently small we obtain
\[
\left(1-\frac{\delta}{R}\right)^n \sim e^{-\frac{n \delta}{R}}
\]
with accuracy estimate following from (\ref{eq:exp_inequality}).}
\end{proof}

In accordance with Proposition \ref{prop:balls_concentration} the volume of an $n$-ball of radius $R$, for $n$ sufficiently large, is concentrated in a thin layer around its surface. Furthermore, in this thin layer the volume of an $n$-ball is concentrated around the largest equator of the corresponding $n-1$ sphere, $\Sphere^{n-1}(R)$.

\begin{prop}\label{prop:disc_concentration} {\color{black} Let $\Ball^n(R)$ be an $n$-ball of radius $R$ in $\Real^n$, and $0<\delta<R$. Let $\Disc^{n-1}_\delta (R)$ be a $\delta$-thickening of an $n-1$-disc $\Disc^{n-1}(R)$. Then}
\[
\frac{\Volume(\Ball^n(R))-\Volume(\Disc_{\delta}^{n-1}(R))}{\Volume(\Ball^n(R))} < {\color{black} e^{-\frac{n \delta^2}{2 R^2}}}.
\]
\end{prop}
\begin{proof} Consider
\[
\begin{split}
&\frac{\Volume(\Ball^n(R))-\Volume(\Ball^n(\sqrt{R^2-\delta^2}))}{\Volume(\Ball^n(R))} =\\
& 1-\left(1-\frac{\delta^2}{R^2}\right)^{\frac{n}{2}}
\end{split}
\]
{\color{black} Using (\ref{eq:exp_inequality}) we obtain}
\[
1-\left(1-\frac{\delta^2}{R^2}\right)^{\frac{n}{2}} >  1 - e^{-\frac{n \delta^2}{2 R^2}},
\]
{\color{black} and the result follows automatically from}
\[
\Volume(\Ball^n(R))-\Volume(\Disc_{\delta}^{n-1}(R))\leq \Volume(\Ball^n(\sqrt{R^2-\delta^2})).
\]
\end{proof}

In high dimension the volume of the ball is concentrated in a thin layer near the sphere. Therefore, the estimate of the volume of the disk automatically provides an estimate of the surface of the corresponding waist of the sphere. Let us produce this estimate explicitly.
The {proportion} of $\Sphere^{n-1}(1)$ belonging to the cap (shaded part of the sphere in Fig. \ref{fig:balls_disks_cones}) equals the proportion
of the solid ball that lies in the corresponding spherical cone (cf. \cite{Ball}, Fig. 11). The latter consist of two parts: one is the cone of height $\delta$ and radius of the base $\sqrt{1-\delta^2}$. The volume of the second part can be bounded from above by the half of the volume of the ball with radius $\sqrt{1-\delta^2}$. If we use the Stirling formula for the volume of high-dimensional ball $V_n(R) \sim \frac{1}{\sqrt{n\pi}}\left(\frac{2\pi e}{n}\right)^{n/2}R^n$ then we obtain that the fraction of the waist of the width $2\delta$ is $1-(1+O(\delta /{\sqrt{n}}))e^{-\frac{n\delta^2}{2}}$. {\color{black} Indeed
\begin{equation}\label{eq:volumes}
\begin{split}
&V_n(R) = \frac{\pi^{n/2}}{\Gamma(\frac{n}{2} + 1)}R^n, \
V_n(R) = 2\pi\left(\frac{ R}{\sqrt{n}}\right)^2 V_{n-2}(R),\\
& V_n(R) = R\sqrt{\pi}\frac{\Gamma(\frac{n+1}{2})}{\Gamma(\frac{n+1}{2} + \frac{1}{2})} V_{n-1}(R).
\end{split}
\end{equation}
The estimate now follows from
$$\Gamma(x)= x^{x - \frac{1}{2}} e^{-x} \sqrt{2\pi} \left( 1 + \frac{1}{12x} + R_2(x)\right),$$
where the reminder $R_2(x)$ can be bounded as
\begin{equation}\label{eq:Stirling_reminder}
|R_2(x)|\leq \frac{1+\frac{1}{6}\pi^2}{2\pi^3 x^2}
\end{equation}
(see (3.11) from \cite{Boyd94} for details).} This estimate improves the textbook estimate $1-2 e^{-\frac{n\delta^2}{2}}$ for large $n$ \cite{Ball}.
\begin{figure}
\begin{center}
\subfigure[Covering of an $n$-ball]{
\includegraphics[width=150pt]{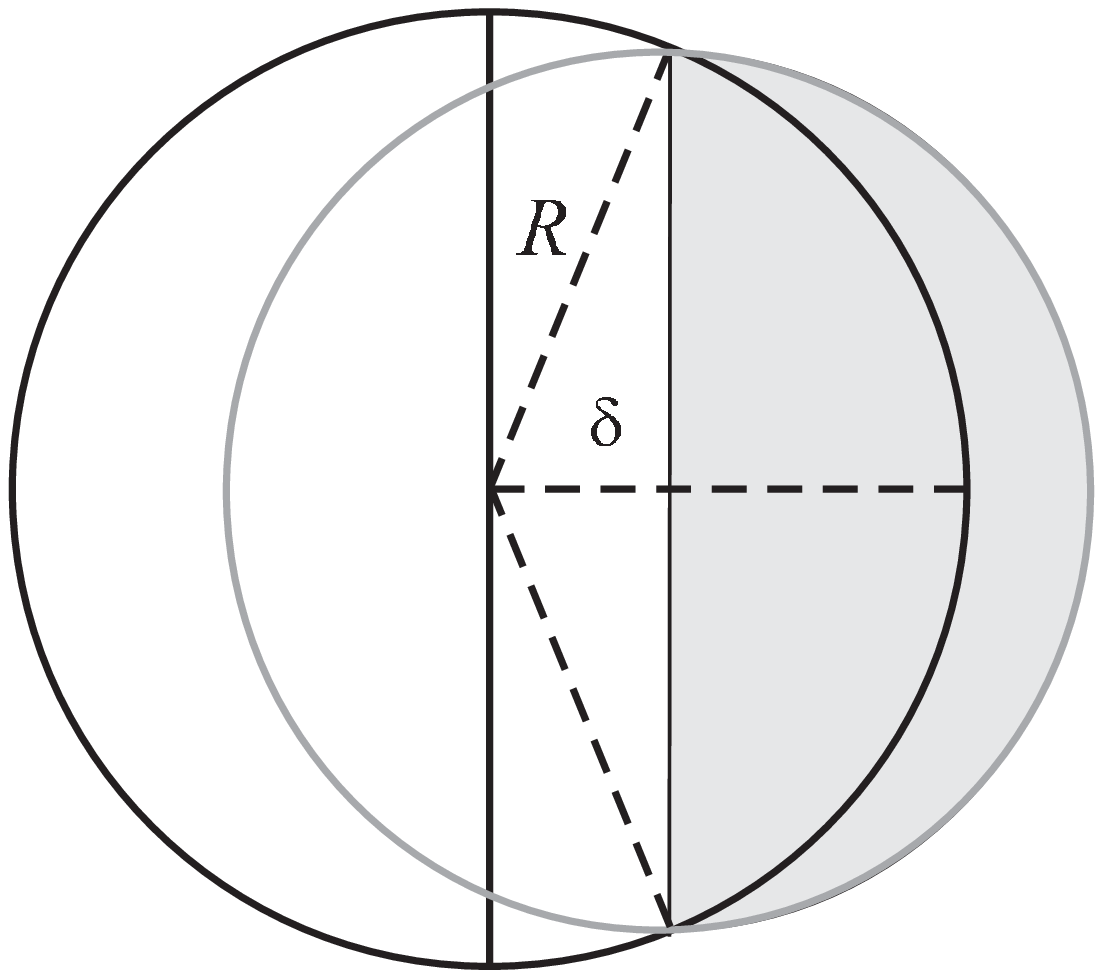}
}\\
\subfigure[$\mathcal{V}\sim\sqrt{\frac{2\pi}{n+1}} R V_{n-1}(R)$]{
\includegraphics[width=100pt]{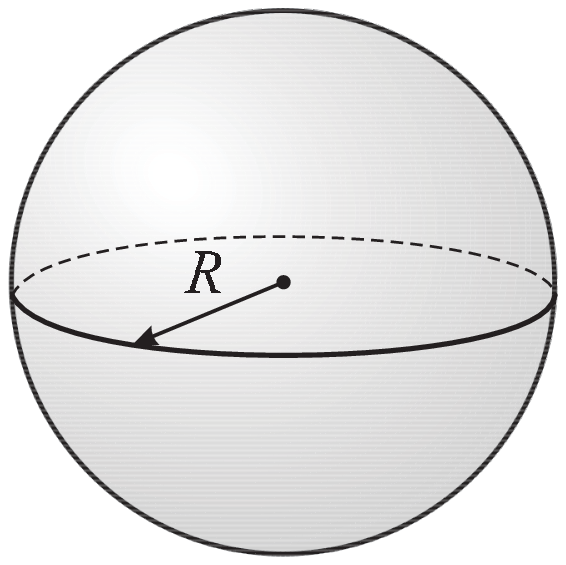}
}
\subfigure[$\mathcal{V}=\frac{1}{n} \delta V_{n-1}(R)$]{
\includegraphics[width=100pt]{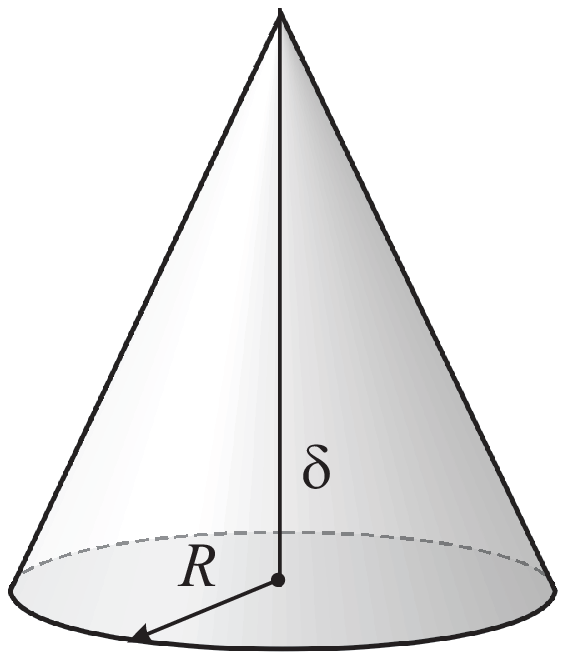}
}
\subfigure[$\mathcal{V}=\delta V_{n-1}(R)$]{
\includegraphics[width=100pt]{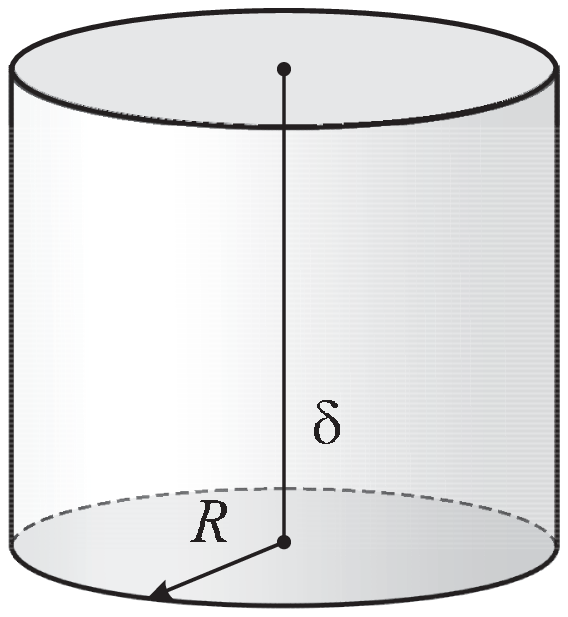}
}
\end{center}
\caption{Illustration to the estimate of the $2\delta$-width of the waist of the sphere. For (b), $n\gg 1$. More accurate estimate with evaluation of the reminder can be extracted from  (\ref{eq:volumes}), (\ref{eq:Stirling_reminder}). }\label{fig:balls_disks_cones}
\end{figure}

{\color{black} The concentration effect described in Proposition \ref{prop:disc_concentration} implies that in high dimensions nearly all independently and randomly drawn vectors will belong to a $\delta$-thickening of a set of vectors that are linearly dependent: the $n-1$-disc $\Disc^{n-1}(R)$.}  On the other hand a set of $n-1$ vectors with probability close to one spans almost all vectors. We note that the latter property holds for systems of $n-k$, $k>1$ vectors too which follows immediately from \cite{Artstein:2002}:

\begin{thm}[Theorem 3.1 in \cite{Artstein:2002}]\label{theorem:artstein} Let $E_k$ be a $k$-dimensional subspace of $\Real^n$, and denote by $\mu((E_k)_\varepsilon)$ the Haar measure on the sphere $\Sphere(1)$ of the set of points within a geodesic distance smaller than $\varepsilon$ of $E_k$. We write $k=\lambda n$. Fix $0<\varepsilon<\pi/2$ and $0<\lambda<1$. The following estimates hold as $n\rightarrow\infty$
\begin{itemize}
\item[(i)] If $\sin^2 \varepsilon > 1 -\lambda$, then
\[
\mu((E_k)_\varepsilon)\simeq 1 - \frac{1}{\sqrt{n\pi}} \frac{\sqrt{\lambda(1-\lambda)}}{\sin^2 \varepsilon - (1-\lambda)} e^{\frac{n}{2}u(\lambda,\varepsilon)}.
\]
\item[(ii)] If $\sin^2 \varepsilon< 1 -\lambda$, then
\[
\mu((E_k)_\varepsilon) \simeq \frac{1}{\sqrt{n\pi}} \frac{\sqrt{\lambda(1-\lambda)}}{(1-\lambda)- \sin^2 \varepsilon}
\]
where $u(\lambda,\varepsilon)=(1-\lambda)\log \frac{1-\lambda}{\sin^2 \varepsilon}+\lambda \log \frac{\lambda}{\cos^2 \varepsilon}$.
\end{itemize}
\end{thm}

Since nearly all vectors in $\Ball^n(1)$ are concentrated in an $\varepsilon$-disc it is interesting to know how many pairwise almost orthogonal vectors can be found in this set. It turns out that this number is exponentially large. Detailed analysis and derivations are provided in the next section.

\subsection{Almost orthogonality in high dimensions}\label{subsec:exponential_dimension}

Select two small positive numbers, $\epsilon$ and $\vartheta$. Let us generate randomly and independently $N$ vectors $x_1,\ldots, x_N$ on $\Sphere^{n-1}(1)$. We are interested in the probability $P$ that all $N$ random vectors are pairwise $\epsilon$-orthogonal, i.e. $|(x_i,x_j)|<\epsilon$ for $i,j=1,\ldots,N$, $i\neq j$. For which $N$ this $P>1-\vartheta$?

Propositions  \ref{prop:balls_concentration}, \ref{prop:disc_concentration} {\color{black} and subsequent discussion} suggest that for $n\gg 1$ almost all volume of an $n$-ball $\Ball^n (1)$ is concentrated in an $\epsilon$-thickening of its largest equator. {\color{black} Moreover} for an arbitrarily chosen point $p$ on the surface of $\Ball^n (1)$ almost all points of the ball belong to {\color{black} the set $\mathcal{DC}^{n-1}_\epsilon(1)$ comprising of the difference} of $\Disc^{n-1}_\epsilon (1)$ {\color{black} with the corresponding spherical cones (see Fig. \ref{fig:balls_disks_cones})}. For the given choice of $p$  and ${\color{black}\mathcal{DC}}^{n-1}_\epsilon(1)$ the following {\color{black} estimate} holds
\[
 |\cos(\angle(p,x))| \leq \epsilon \ \forall \ x\in {\color{black} \mathcal{DC}}^{n-1}_\epsilon(1).
\]
The latter property {\color{black} means} that the vector $p$ is almost orthogonal to {\color{black} nearly all remaining} points in $\Ball^n (1)$.

{\color{black} Another way of looking at this could be as follows.} In accordance with Proposition \ref{prop:balls_concentration}, almost all points of the ball $\Ball^n(1)$ are concentrated around a $\varepsilon$-thickening of surface. {\color{black} At the same time they are also concentrated in $\Disc^{n-1}_\epsilon (1)$}. Lengths of such points, $x$, satisfy $1-\varepsilon\leq \|x\|\leq 1$, and hence
\[
|\cos(\angle(p,x))|\leq\frac{|p^T x|}{1 -\varepsilon}\leq \frac{\epsilon}{1-\varepsilon}.
\]

Let us determine the number of independent random vectors which are pairwise $\epsilon$-orthogonal with probability $1 - \vartheta$. The volume taken by all vectors that are almost orthogonal to a given vector on a unit sphere can be estimated from Proposition \ref{prop:disc_concentration} (up to a small correction term $O({\delta e^{-\frac{n\delta^2}{2}}}/{\sqrt{n}})$). Consider the following products
\begin{equation}\label{eq:almost_orthogonal}
P(\epsilon,N)=\prod_{k=1}^{N}\left(1 - k e^{-\frac{n \epsilon^2}{2}}\right).
\end{equation}
The value of $P(\epsilon, N)$ is an estimate from below of the probability of a set of $N+1$ independent random vectors to be pairwise $\epsilon$-orthogonal. Indeed, for one vector $h_1$ the fraction of the vectors which are {\it not} $\epsilon$-orthogonal to $h_1$ is evaluated as $e^{-n\delta^2}$. Therefore, for $k$ vectors $h_1, \ldots, h_k$, the fraction of the vectors which are not  $\epsilon$-orthogonal to $h_1, \ldots, h_k$ is at most $ke^{-n\delta^2}$. The probability to select randomly a vector $h_{k+1}$, which is $\epsilon$-orthogonal to $h_1, \ldots, h_k$, is higher than $1-ke^{-n\delta^2}$. The vectors are selected independently, therefore we have the estimate (\ref{eq:almost_orthogonal}).


{\color{black} The value of $P(\epsilon,N)$ in (\ref{eq:almost_orthogonal}) can be estimated as follows}. For
$N e^{-\frac{n \epsilon^2}{2}} < 1$:
\begin{equation}\label{eq:almost_orthogonal:1}
\begin{split}
P(\epsilon,N)>\left(1 - N e^{-\frac{n \epsilon^2}{2}}\right)^N & \sim e^{-N^2 e^{-\frac{n\epsilon^2}{2}}}.
\end{split}
\end{equation}
According to (\ref{eq:almost_orthogonal:1}), if $P(\epsilon,N)$ is set to be exceed a certain value, $1-\vartheta$, the number of pairwise almost orthogonal vectors in $\Ball^n (1)$ will have the following asymptotic estimate: For
\begin{equation}\label{eq:exponential_dimension}\boxed{
 N\leq e^{\frac{\epsilon^2 n}{4}} \left[\log\left(\frac{1}{1-\vartheta}\right)\right]^{\frac{1}{2}}}
\end{equation}
all random vectors $h_1,\ldots, h_{N+1}$ are pairwise $\epsilon$-orthogonal with probability $P>1-\vartheta$.

Estimate (\ref{eq:almost_orthogonal:1}) of (\ref{eq:almost_orthogonal}) can be refined if we apply $\log $ to the right hand side of (\ref{eq:almost_orthogonal})
\[
\log P(\epsilon,N) = \sum_{k=1}^{N} \log \left(1 - k e^{-\frac{n \epsilon^2}{2}} \right),
\]
end estimate the above sum using the integral
\[
J(z)=\int_{0}^{z} \log (1- x r) dx=\frac{rz-1}{r} \log(1-r z) - z , \ r= e^{-\frac{n \epsilon^2}{2}}
\]
Since $\log P(\epsilon,N)$ is monotone for $e^{\frac{n \epsilon^2}{2}}>N\geq 1$ we can conclude that $J(N+1) \leq \log P(\epsilon,N) \leq J(N)$. Furthermore, given that
\[
 \log(1-x)\leq -\frac{2x}{2-x} \ \mathrm{for} \ \mathrm{all} \ x\in[0,1]
 \]
 the following holds
\[
J(z) \geq  \frac{2z(rz-1)}{rz-2} - z
\]
for all $zr\in[0,1]$. Hence
\[
\frac{2 (N+1) (r (N+1)-1)}{r (N+1)-2} - (N+1) \leq P(\epsilon,N)=\log (1-\vartheta).
\]
Transforming this equation into quadratic
\[
r (N+1)^2 - \log(1-\vartheta) r (N+1) + 2 \log (1-\vartheta)\leq 0
\]
and solving for $N$ results in
\begin{equation}\label{eq:improved_estimate}
\boxed{N\leq \sqrt{\frac{\log^2(1-\vartheta)}{4} + 2 \log\frac{1}{1-\vartheta} e^{\frac{n \epsilon^2}{2}}} + \frac{\log(1-\vartheta)}{2}}
\end{equation}
Notice that the refined estimate (\ref{eq:improved_estimate}) has asymptotic exponential rate of order $e^{\epsilon^2n/4}$  which is identical to the one derived in (\ref{eq:exponential_dimension}).

\section{Discussion}\label{sec:discussion}

Estimate (\ref{eq:exponential_dimension}), (\ref{eq:improved_estimate}) derived in the previous section suggests that, for  $\vartheta$ sufficiently small  a set of $N$ randomly and independently chosen vectors in $\Ball^n(1)$ will be pairwise $\epsilon$-orthogonal with probability $1-\vartheta$ for
\[
N\lesssim  e^{\frac{\epsilon^2 n}{4}} \vartheta^{\frac{1}{2}}
\]
Such exponential bound enables us to explain some apparent controversy \cite{Tyukin:2009} regarding convergence rates of approximation schemes based on iterative greedy approximation \cite{Jones:1992}, \cite{Barron}, and randomized choice of functions advocated in \cite{igelpao95}, \cite{Rahimi08}, \cite{Rahimi08:2}.

Both greedy approximation and systems of randomized basis functions enjoy the convergence rate of order $1/N^{1/2}$ in $L_2$-norm irrespective of dimensionality of the domain on which a function that is being approximated is defined. Greedy approximation, however, requires solving nonlinear optimization problems at each step. In randomized {\color{black} strategies}, parameters of the basis functions/kernels are randomly drawn from a given set. This transforms the original problem into a much simpler linear one.

Notice, however, that the constant $C$ in the error rate  $O(C/N^{1/2})$ can become rather large in these schemes. If, for example kernels $\phi(x,\omega)$ in Theorem \ref{theorem:Rahimi} approach  Dirac delta-functions, then the constant $C$ in (\ref{eq:error_random:2}) becomes
arbitrarily large. The same problem may occur for the error rate (\ref{eq:error_random:1.5}). This means that practical application of randomized approximation schemes is to be preceded by additional analysis of the class of functions being approximated. Second,  achievable rate in these schemes is in general dependent on the values of $\delta$, determining probability of success (cf. (\ref{eq:error_random:2}); see our comment on Monte-Carlo convergence in probability too). The smaller the value of $\delta$ the higher is the probability of achieving the rate $O(C/N^{1/2})$. That rate, however, is proportional to a strictly monotone function of $1/\delta$. This might be acceptable in many practical applications. Nevertheless, if particularly high accuracy is desirable one needs to take this into account.

An insight into the slowing-down of convergence of randomized approximators due to either increased $C$ or due to the need to keep $\delta$ small can be gained by considering measure concentration effects discussed in Propositions \ref{prop:balls_concentration}, \ref{prop:disc_concentration}, Theorem \ref{theorem:artstein} and in Section \ref{subsec:exponential_dimension}. Indeed, in accordance with the analysis presented a system of $n-k$, $k\geq 1$ vectors spans almost all vectors in a ball $\Ball^n(1)$. This system, however, is almost linearly dependent. This means that the coefficients in corresponding linear combinations could become arbitrary large. Large coefficients, in turn, make the representation sensitive to small inaccuracies which is of course generally not very desirable. This observation of typicality of large coefficients in randomized choice of relatively short bases is consistent with our earlier remarks about growth of $C$ and influence of $1/\delta$.
On the other hand, if a representation of data is needed in which all coefficients are to be within certain bounds then the number of randomized basis vectors may become exponentially large.

To overcome potential danger of exponential growth of the number of elements needed to achieve acceptable level of performance of a randomized approximator one may constrain dimensionality of randomization by ``binning'' or partitioning the space of original problem into a set of spaces of smaller {\color{black} dimensions}. An example of such an optimization would be to use e.g. multi-scale basis functions, followed by randomization at each scale. Cascading, multi-grid, frequency separation, localization could all be used for the same purposes. The achieved efficiency will of course be determined by suitability of each of these ``binning'' approaches for a problem at hand. {\color{black} Of course such a binning will violate isotropy of the space. Violation of isotropy can make random approximation much more efficient. For example \cite{Tao} proved that  if the vectors of signals are sparse in a {\it fixed basis}, then it is possible to reconstruct these signals  with very high accuracy from a small
number of random measurements.}

\section{Examples}\label{sec:Example}

\subsection{Error convergence  rate. Deterministic vs Randomized approaches}

In order to illustrate the main difference between greedy and RVFL
approximators, we consider the following example in which a simple
function is approximated by both methods, greedy approximation
(\ref{eq:approximation})--(\ref{eq:Jones_iteration}) and
approximation based on randomized choice of bases. Let $f(x)$ be
defined as follows:
\[
f(x)=0.2 e^{-(10x-4)^2} + 0.5e^{-(80x-40)^2}+0.3e^{-(80x-20)^2}
\]
The function $f(x)$ is shown in Fig. \ref{fig:example}, top
panel. Clearly, $f(\cdot)$
belongs to the convex hull of $G$, and hence to its closure.

First, we implemented greedy approximation
(\ref{eq:approximation})--(\ref{eq:Jones_iteration}) in which we
searched for $g_n$ in the following set of functions
\[
{\color{black}\mathcal{G}}=\{e^{-(w^T x +
b)^2}\},
\]
where $w\in[0,200]$, $b\in[-100,0]$. The procedure for constructing $f_n$ was as follows.
Assuming $f_0(x)=0$, $e_0=-f$  we started with searching for $w_1$, $b_1$ such that
\begin{equation}\label{eq:exaple_greedy:1}
\begin{split}
&\langle0-f(x),g(w_1x+b)-f(x)\rangle=\\
&-\langle f(x),g(w_1x+b)\rangle+\|f(x)\|^2 < \varepsilon.
\end{split}
\end{equation}
where $\varepsilon$ was set to be small ($\varepsilon=10^{-6}$ in
our case). When searching for a solution of
(\ref{eq:exaple_greedy:1}) (which exists because the function $f$ is
in the convex hull of $G$ \cite{Jones:1992}), we did not utilize any
specific optimization routine. We sampled the space of parameters
$w_i$, $b_i$ randomly and picked the first values of $w_i$, $b_i$
which satisfy (\ref{eq:exaple_greedy:1}). Integral
(\ref{eq:exaple_greedy:1}) was evaluated in quadratures over a
uniform grid of $1000$ points in $[0,1]$.

The values of $\alpha_1$ and the function $f_1$ were chosen in
accordance with (\ref{eq:Jones_iteration}) with $M''=2$, $M'=1.5$
(these values are chosen to assure $M''>M'> \sup_{{\color{black} g\in\mathcal{G}}} \|g\|+\|f\|$).
The iteration was repeated, resulting in the following sequence of
functions
\[
\begin{split}
f_n(x)&=\sum_{i=1}^{\color{black}N} c_i g(w_i^{T}x+b_i), \\ c_i&=\alpha_i(1-\alpha_{i+1})(1-\alpha_{i+2})\cdots(1-\alpha_{{\color{black} N}})
\end{split}
\]

Evolution of the normalized approximation error
\begin{equation}\label{eq:exaple_error_normalized}
\bar{e}_{\color{black}N}=\frac{e_{\color{black}N}^2}{\|f\|^2}=\frac{\|f_{\color{black} N}-f\|^2}{\|f\|^2}
\end{equation}
for $100$ trials is shown in Fig. \ref{fig:example} (middle panel).
Each trial consisted of $100$ iterations
(\ref{eq:approximation})--(\ref{eq:Jones_iteration}), thus leading
to the networks of $100$ elements at the $100$th step. We observe
that the values of $\bar{e}_{\color{black} N}$ monotonically decrease as $O(1/{\color{black} N})$,
with the behavior of this approximation procedure consistent across
trials.


Second, we implemented an approximator based on the Monte-Carlo
integration.  At the ${\color{black} N}$th step of the approximation procedure we
pick randomly an element from $G$, where $w\in[0,200]$,
$b\in[-200,200]$ (uniform distribution). After an element is
selected, we add it to the current pool of basis functions
\[
\{g(w_1^Tx+b_1),\dots,g(w_{{\color{black} N}-1}x+b_{{\color{black} N}-1})\}.
\]
Then the weights $c_i$ in the superposition
\[
f_{\color{black} N}=\sum_{i=1}^{\color{black} N} c_i g(w_i^Tx+b_i)
\]
are optimized so that $\|f_{\color{black} N}-f\|\rightarrow\min$. Evolution of the
normalized approximation error $\bar{e}_{\color{black} N}$
(\ref{eq:exaple_error_normalized}) over $100$ trials is shown in
Fig. \ref{fig:example} (panel (c)). As can be observed from the
figure,  even though the values of $\bar{e}_{\color{black} N}$ form a monotonically
decreasing sequence, they are far from $1/{\color{black} N}$, at least for $1\leq
{\color{black}N}\leq 100$. Behavior across trials is not consistent, at least for
the networks smaller than $100$ elements, as indicated by a
significant spread among the curves.

\begin{figure}
\begin{center}
\includegraphics[width=150pt]{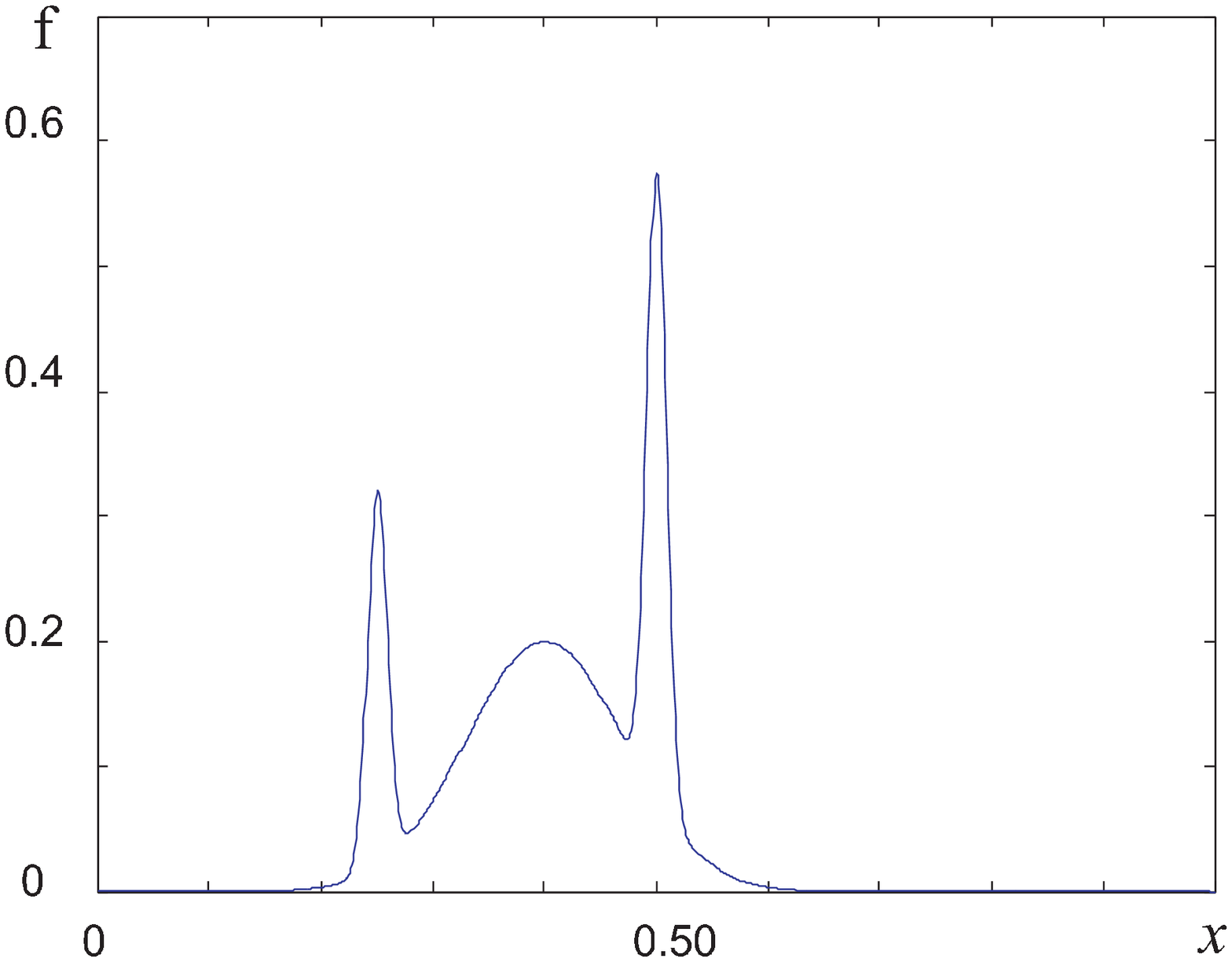} \hspace{5mm}
\includegraphics[width=155pt]{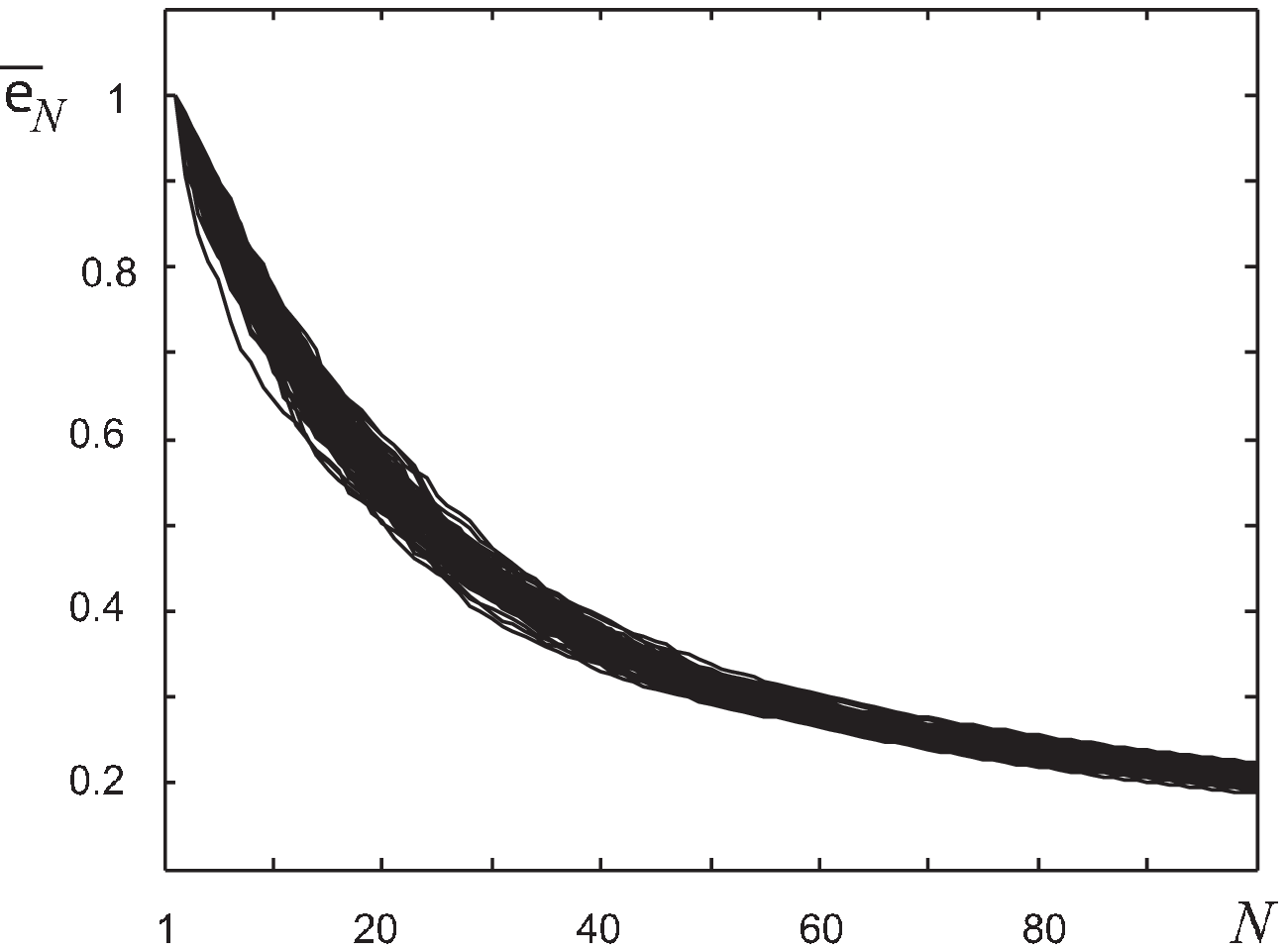}

\small (a) \hspace{6.5cm} (b)
\end{center}

\begin{center}
\includegraphics[width=155pt]{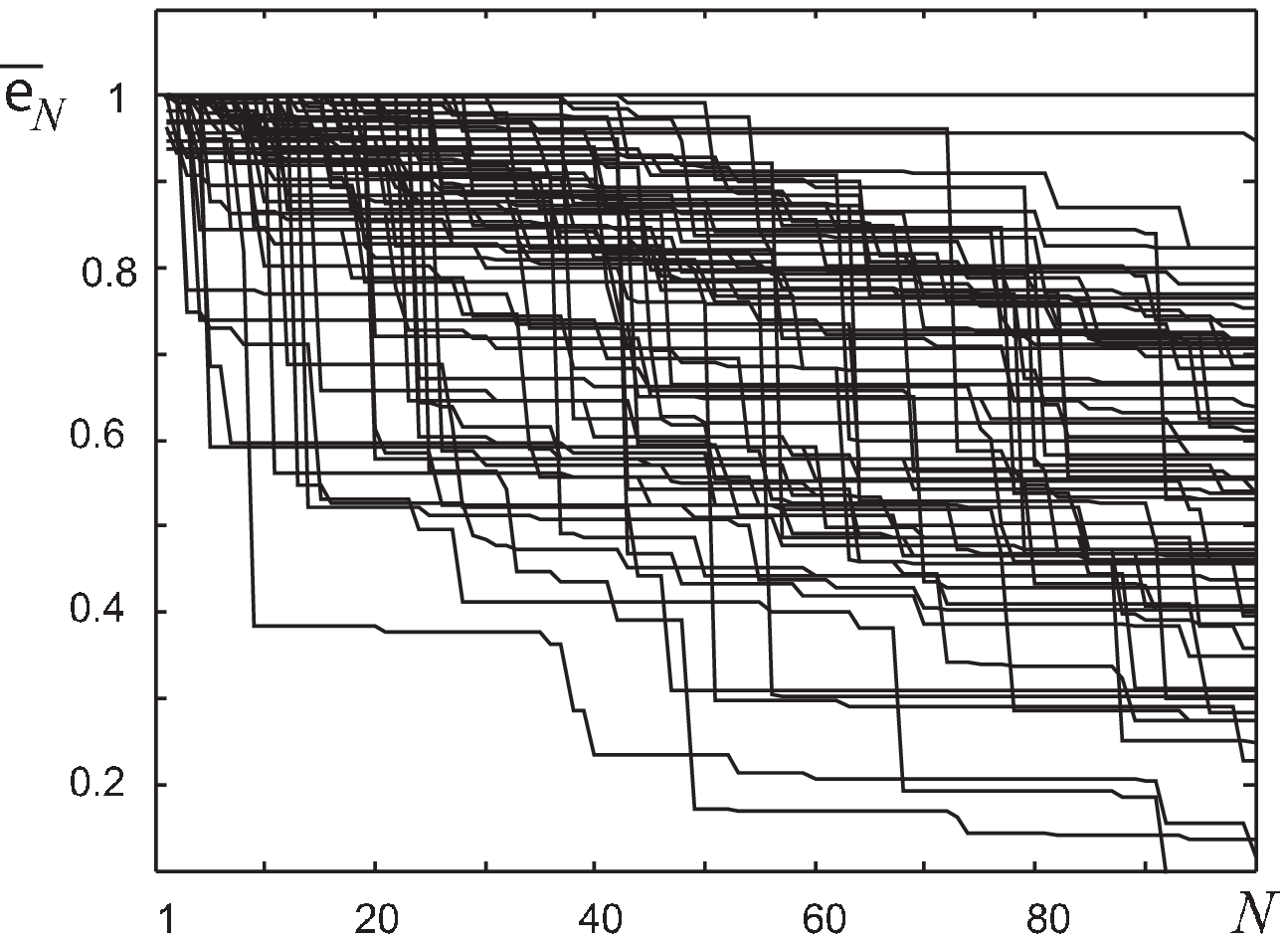}

\small (c)
\end{center}
\caption{Practical speed of convergence of function approximators
that use greedy algorithm (panel (b)) and Monte-Carlo based
random choice of basis functions (panel (c)). The target function
is shown in panel (a). }\label{fig:example}
\end{figure}

Overall comparison of these two methods is provided in Fig.
\ref{fig:example:2}, in which the errors $\bar{e}_{\color{black} N}$ are presented
in the form of a box plot. Black solid curves depict the median of
the error as a function of the number of elements, ${\color{black}N}$, in the
network; blue boxes contain $50\%$ of the data points in all trials;
``whiskers'' delimit the areas containing $75\%$ of data, and red
crosses show the remaining part of the data. As we can see from
these plots, random basis function approximators, such as the RVFL
networks, mostly do not match performance of greedy approximators
for networks of reasonable size. Perhaps, employing integration
methods with variance minimization could improve the performance.
This, however, would amount to using prior knowledge about the
target function $f$, making it difficult to apply the RVFL networks
to problems in which the function $f$ is uncertain.

\begin{figure}
\begin{center}
\includegraphics[width=0.45\columnwidth]{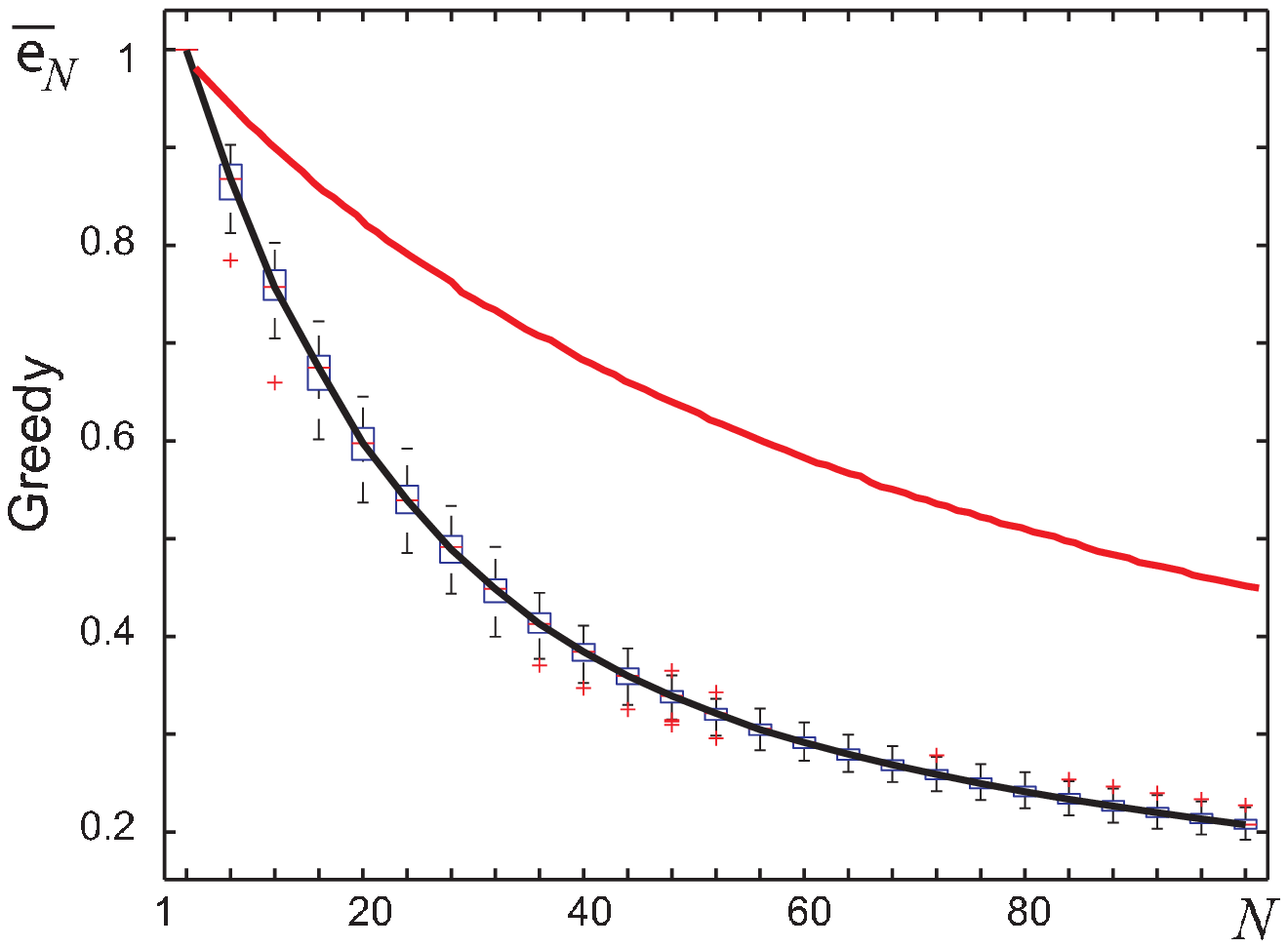} \hspace{3mm}
\includegraphics[width=0.45\columnwidth]{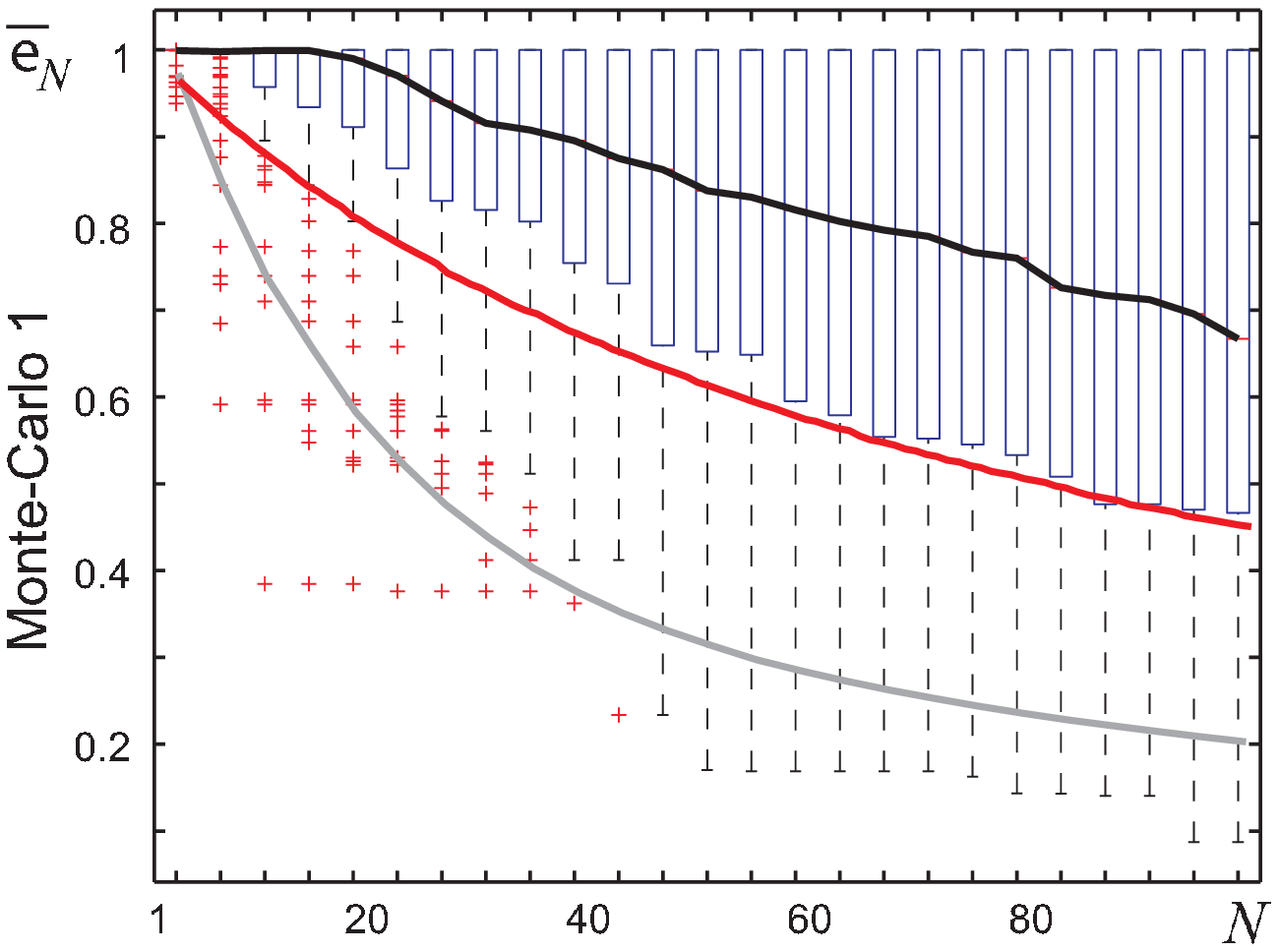}

\small (a) \hspace{6.5cm} (b)
\end{center}
\begin{center}
\includegraphics[width=0.45\columnwidth]{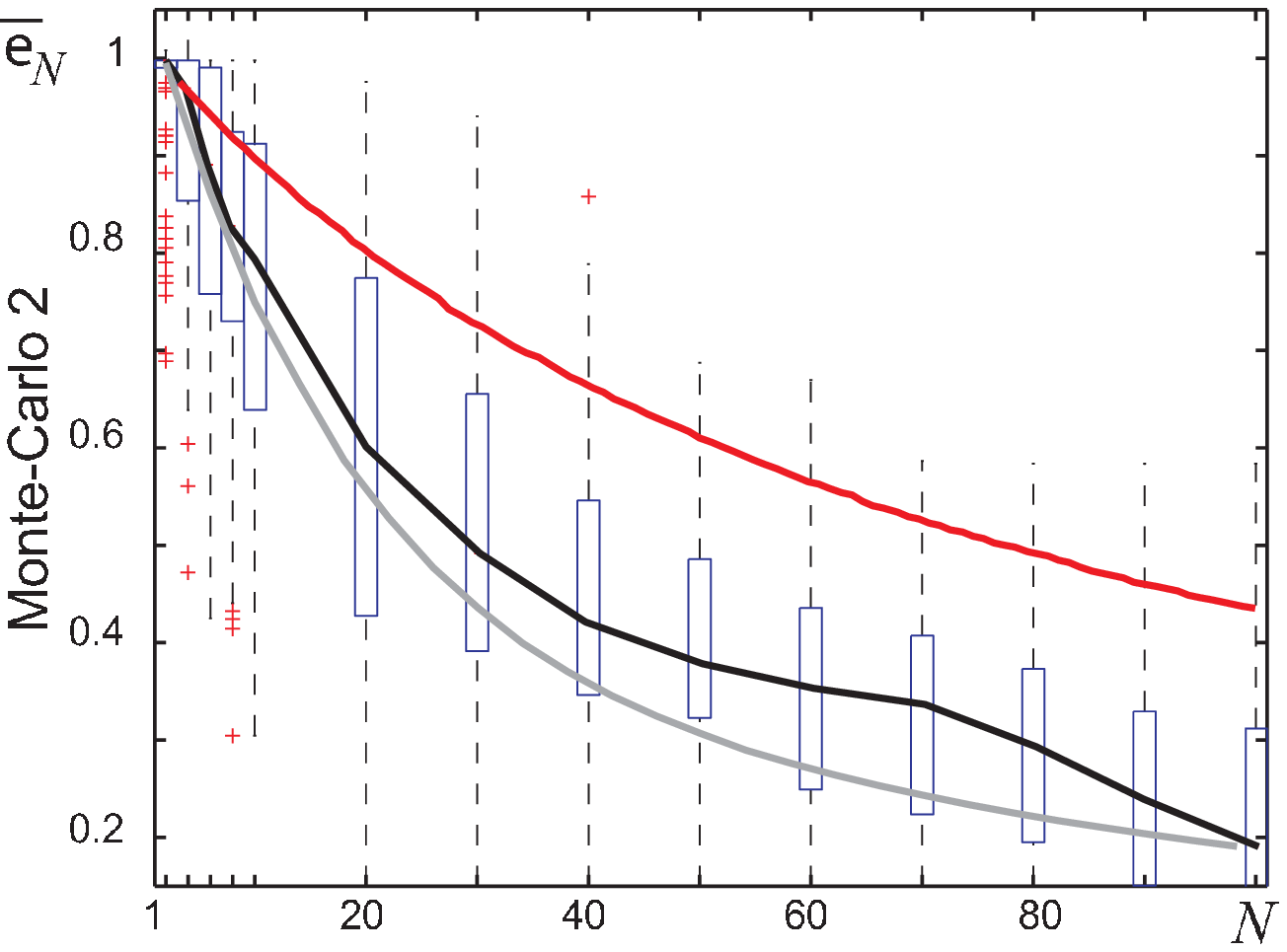}

(c)
\end{center}
\caption{Box plots of convergence rates for function approximators
that use greedy algorithm (panel (a)) and Monte-Carlo random choice
of basis functions (panels (b) and (c)). Panel (b)
corresponds to the case in which the basis functions leading to
ill-conditioning were discarded. Panel (c) shows performance
of the MLP trained by the method in \cite{Feldkamp98b} which is
effective at counteracting ill-conditioning while adjusting the
linear weights only. The red curve shows the upper bound for
$\bar{e}_{\color{black} N}$ calculated in accordance with (\ref{eq:rate:greedy}). We
duplicated the average performance of the greedy algorithm (grey
solid curve) in the middle and bottom panels for convenience of
comparison.}\label{fig:example:2}
\end{figure}

Now we demonstrate performance of an MLP trained to approximate this
target function. The NN is trained by a gradient based method
described in \cite{Feldkamp98b}. At first, the full network training
is carried out for several network sizes ${\color{black} N}=20$, $40$, $60$, $80$
and $100$ and input samples randomly drawn from $x\in [0,1]$. The
values of $\bar{e}_{\color{black}N}$ are $1.5\cdot10^{-4}$ for all the network
sizes (as confirmed in many training trials repeated to assess
sensitivity to weight initialization). This suggests that training
and performance of much smaller networks should be examined. The
networks with ${\color{black} N}=2,4,6,8,10$ are trained, resulting in
$\bar{e}_{\color{black} N}=0.5749,0.1416,0.0193,0.0011,0.0004$, respectively,
averaged over $100$ trials per the network size.  Next, we train
only the linear weights ($c_i$ in (\ref{eq:mlp})) of the MLP, fixing
the nonlinear weights $w_i$ and $b_i$ to random values. The results
for $\bar{e}_{\color{black} N}$ averaged over $100$ trials are shown in Fig.
\ref{fig:example:2}, bottom panel (black curve). Remarkably, the
results of random basis network with ${\color{black} N}=100$ are worse than those of
the MLP with ${\color{black}N}\ge 4$ and full network training.  These results
indicate that both the greedy and the Monte-Carlo approximation
results shown in Fig. \ref{fig:example:2} are quite conservative.
Furthermore, the best of those two, i.e., the greedy
approximation's, can be dramatically improved by a practical
gradient based training.

\addtolength{\textheight}{-1cm}   

\subsection{Measure concentration effects}

Measure concentration effects, as presented in Section \ref{sec:main}, have been discussed for idealized objects such as $\Sphere^{n-1}(R)$ and $\Ball^n(R)$. The phenomenon, however, broadly applies to other objects whose geometric and formal description is not limited to the former.

In order to illustrate this point we analysed a database of HOG feature vectors \cite{HOG} containing representations  of images of faces\footnote{The database has been developed by Apical LTD.} as well as the negatives (non-faces). Each feature vector has $1920$ components, and hence belongs to $\Real^n$ with $n=1920$. Vectors of each classes have been centered and normalized so that they belong to the hypercube $[-1,1]^n$. Fig. \ref{fig:HoG_angles} shows distributions of angles between a randomly chosen vector ($1$-st) and that of the rest in their respective classes.
\begin{figure}
\begin{center}
\includegraphics[width=300pt]{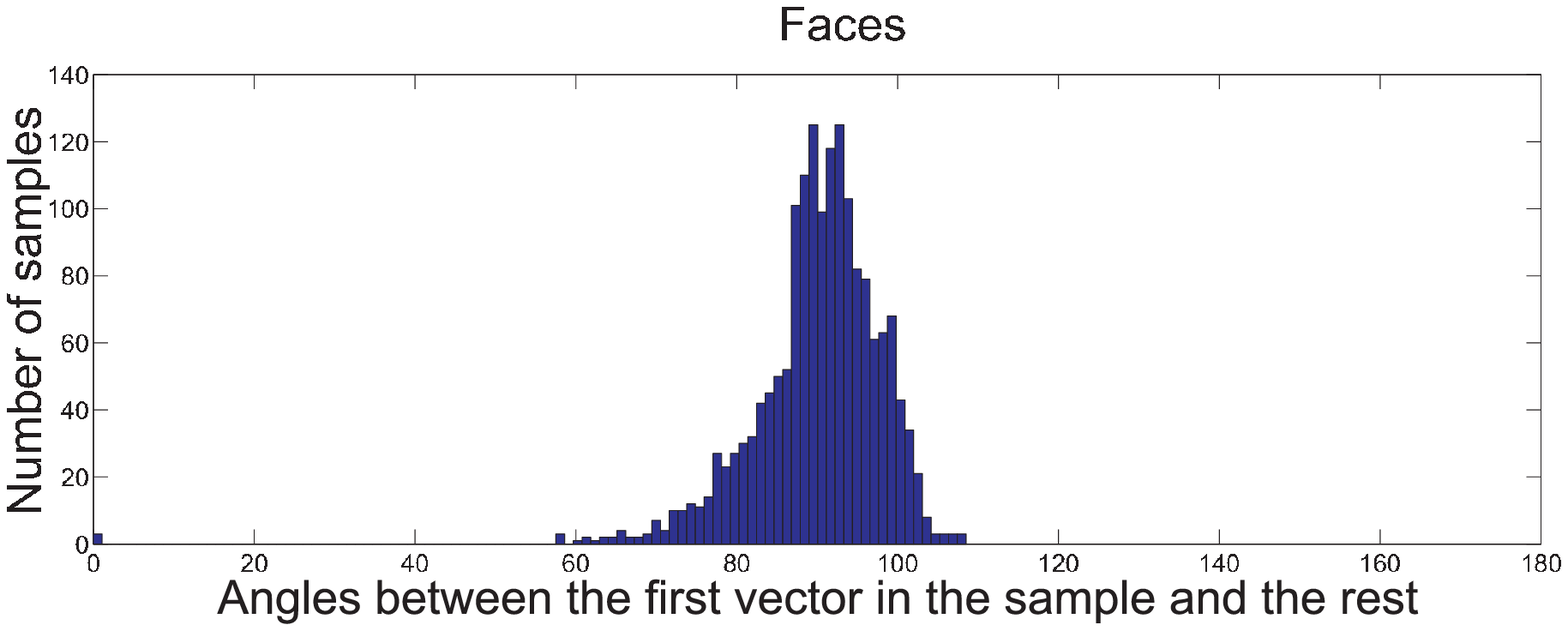}
\small (a)
\end{center}

\begin{center}
\includegraphics[width=300pt]{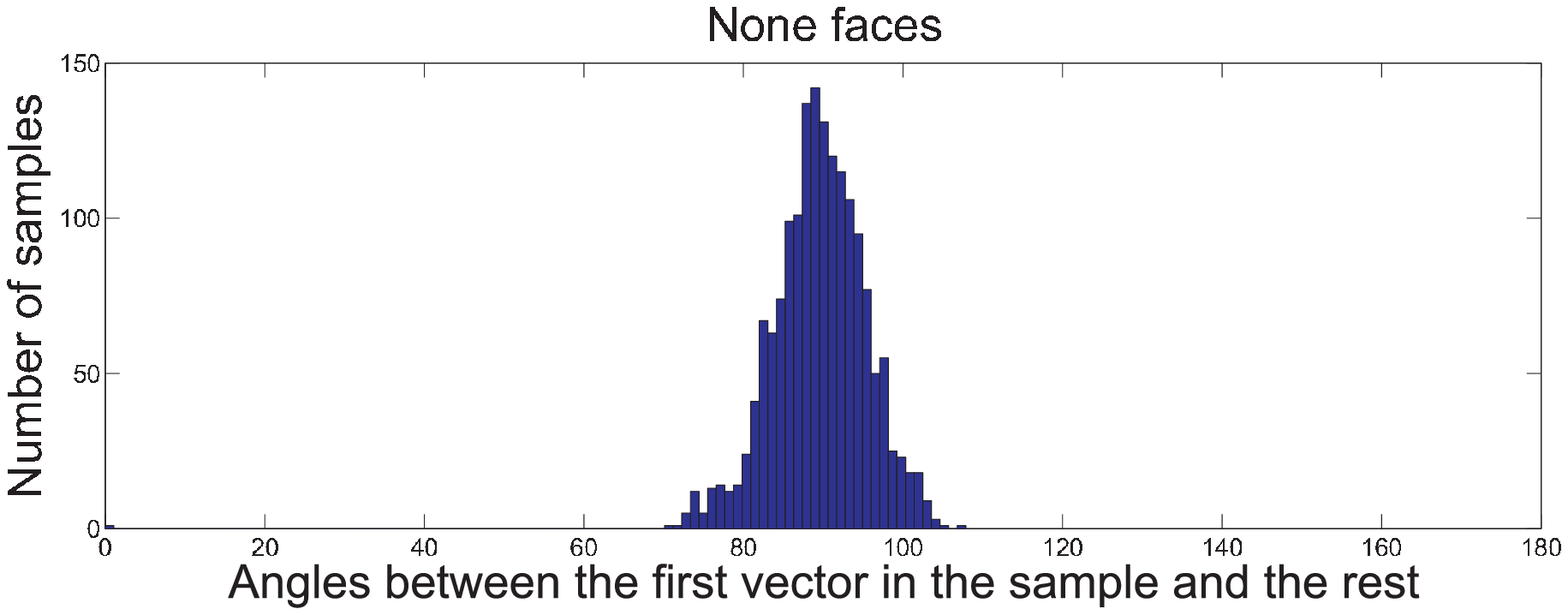}

\small (b)
\end{center}
\caption{Measure concentration in high dimensions. Panel (a) shows histogram of angles between a randomly chosen feature vector in the set of "faces" and the rest of the vectors in the class. Panel (b) shows histogram of angels between a randomly chosen feature vector in the set of "non-faces" (negatives) and the rest of vectors in this class.}\label{fig:HoG_angles}
\end{figure}
As one can see from this figure, the angles concentrate in a vicinity of $\pi/2$ which is consistent with our derivations for points in $\Ball^{n}(1)$.

Another interesting effect of measure concentration is exponential growth of the lengths of chains of randomly chosen vectors which are pairwise almost orthogonal. In order to illustrate and assess validity of our estimates (\ref{eq:exponential_dimension}), (\ref{eq:improved_estimate}) the following numerical experiments have been performed. A point is first randomly selected in a hypercube $[-1,1]^n$ of some given dimension. The second point is randomly chosen in the same hypercube. These two points correspond to two vectors randomly drawn in $[-1,1]^n$. If the angle between the vectors was within $\pi/2\pm 0.037\pi/2$ then the vector was retained. At the next step a new vector is generated in the same hypercube, and its angles with the previously generated vectors are evaluated. If these angles are within $0.037\pi/2$ of $\pi/2$ then the vector is retained. The process is repeated until the chain of almost orthogonality breaks, and the number of such pairwise almost orthogonal vectors (length of the chain) is recorded. Results are shown in Figure \ref{fig:dimensionality_bounds}. Red line corresponds to the conservative theoretical estimate (\ref{eq:exponential_dimension}), green curve shows refined estimate, (\ref{eq:improved_estimate}), and box plot shows lengths of pairwise almost orthogonal chains as a function of dimension. The value of $\vartheta$ was set to $0.1$ for both theoretical estimates, and our choice of precision margins $\pi/2\pm 0.037\pi/2$ {\color{black} corresponds} to $\varepsilon=\cos(0.963\pi/2)=0.0581$. As we can see from this figure our empirical observations are well aligned with theoretical predictions.

\begin{figure}
\begin{center}
\includegraphics[width=300pt]{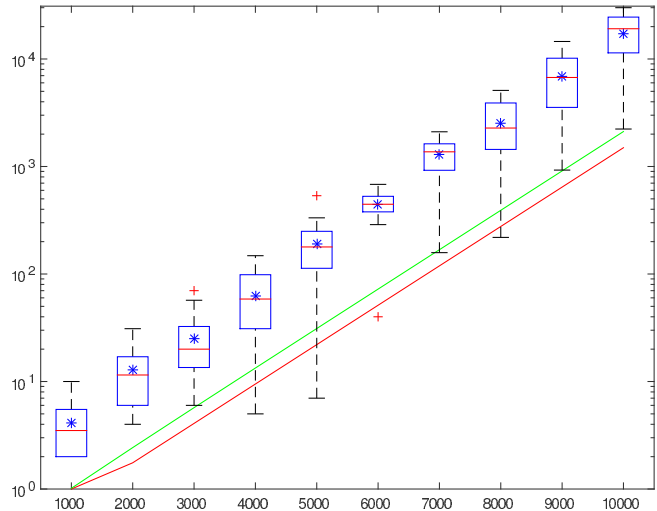}
\end{center}
\caption{Lengths $N$ of pairwise almost orthogonal chains of vectors that are independently randomly sampled from $[-1,1]^n$ as a function of dimension, $n$. For each $n$ $20$ pairwise almost orthogonal chains where constructed numerically. Boxplots show the second and third quartiles of this data for each $n$, red bars correspond to the medians, and blue stars indicate means. Red curve shows theoretical bound (\ref{eq:exponential_dimension}), and green curve shows refined estimate (\ref{eq:improved_estimate}).}\label{fig:dimensionality_bounds}
\end{figure}

\subsection{Approximation of a constant: dimensionality blowup}

Another, and perhaps, more interesting illustration of measure concentration and orthogonality effects belongs to the field of function approximation. Let us suppose that we are to approximate a given continuous function defined on an interval $[0,1]$ by linear combinations of the following type
\begin{equation}\label{eq:approx_lin}
f_N(x)=\sum_{i=1}^N c_i \varphi(a_i,\sigma_i,x),
\end{equation}
where the function $\varphi$ is defined as follows
\[
\varphi(a,\sigma,x)=\left\{\begin{array}{ll} 0, & x>a+\sigma/2\\
                                            1, & a-\sigma/2\leq x\leq a+\sigma/2\\
                                            0, & x < a-\sigma/2 \end{array}\right.
\]
For simplicity we suppose that the function to be approximated, $f^\ast$ is a constant:
\[
f^\ast(x)=1 \ \forall \ x\in[0,1].
\]

Linear combinations of $\varphi(a_i,\sigma_i,\cdot)$  can uniformly approximate every continuous function on $[0,1]$. Furthermore the chosen function $f^\ast$ can be represented by just a single element with $a=0.5,\sigma=0.5$: $f^\ast(x)=\varphi(0.5,0.5,x)$. Since we assumed no prior knowledge of these parameters, we approximated the function $f^\ast$ with linear combinations (\ref{eq:approx_lin}) in which the values of $a_i,\sigma_i$ were chosen randomly in the interval $[0,1]$, and the values of $c_i$ were chosen as follows
\begin{equation}\label{eq:example_approx_opt}
c_1,\dots,c_N=\arg \min_{c_1,\dots,c_N} \int_0^1 (f^\ast(x)-f_N(x))^2 dx.
\end{equation}

In order to evaluate performance of approximation as a function of $N$ the following iterative procedure has been used. On the first step the values of $a_1$ and $\sigma_1$ are randomly drawn from the interval $[0,1]$. This is followed by finding the  value of $c_1$ in accordance with (\ref{eq:example_approx_opt}) (it is clear though that $c_1=1$). Next the values of $a_2,\sigma_2$ are drawn from $[0,1]$ followed by determination of optimal weights $c_1$, $c_2$. The $L_2$ error of approximation is recorder for each step. The process repeats until $N=500$.

Fig. \ref{fig:ensemble_errors} presents $20$ different error curves corresponding to different growing systems of functions $\{\varphi(a_i,\sigma_i,\cdot{\color{black})}\}_{i=1}^{N}$.
\begin{figure}
\includegraphics[width=300pt]{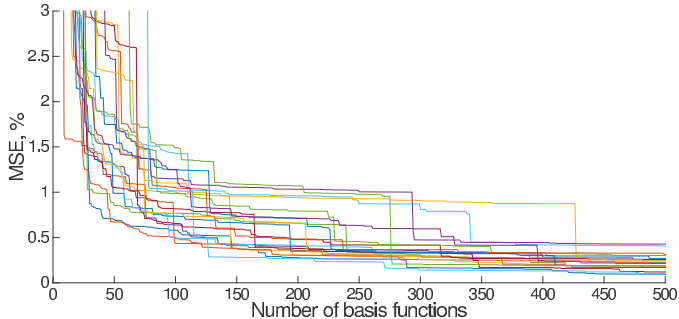}
\caption{Errors of approximation of $f^\ast(x)=1$ by linear combinations $f_N(x)$, (\ref{eq:approx_lin}), as functions of $N$.} \label{fig:ensemble_errors}
\end{figure}
Despite that {\color{black} the} problem is both a) simple and b) admits explicit solutions with respect to the values of $c_1,\dots,c_N$, performance of such approximations in terms of convergence rates is far from ideal. One can observe that initially there is a significant drop of the approximation error for $N\leq 100$. After this value, however, the error decays very slowly and its rate of decay nearly stalls for $N$ large.

One of potential explanation of this effect is as follows. Fig. \ref{fig:errors_approx_individual} shows functions $f^\ast(x)-f_N(x)$ for $N=5$, $50$, and $500$ along a single typical curve from Fig. \ref{fig:ensemble_errors}.
\begin{figure}
\begin{center}
\includegraphics[width=300pt]{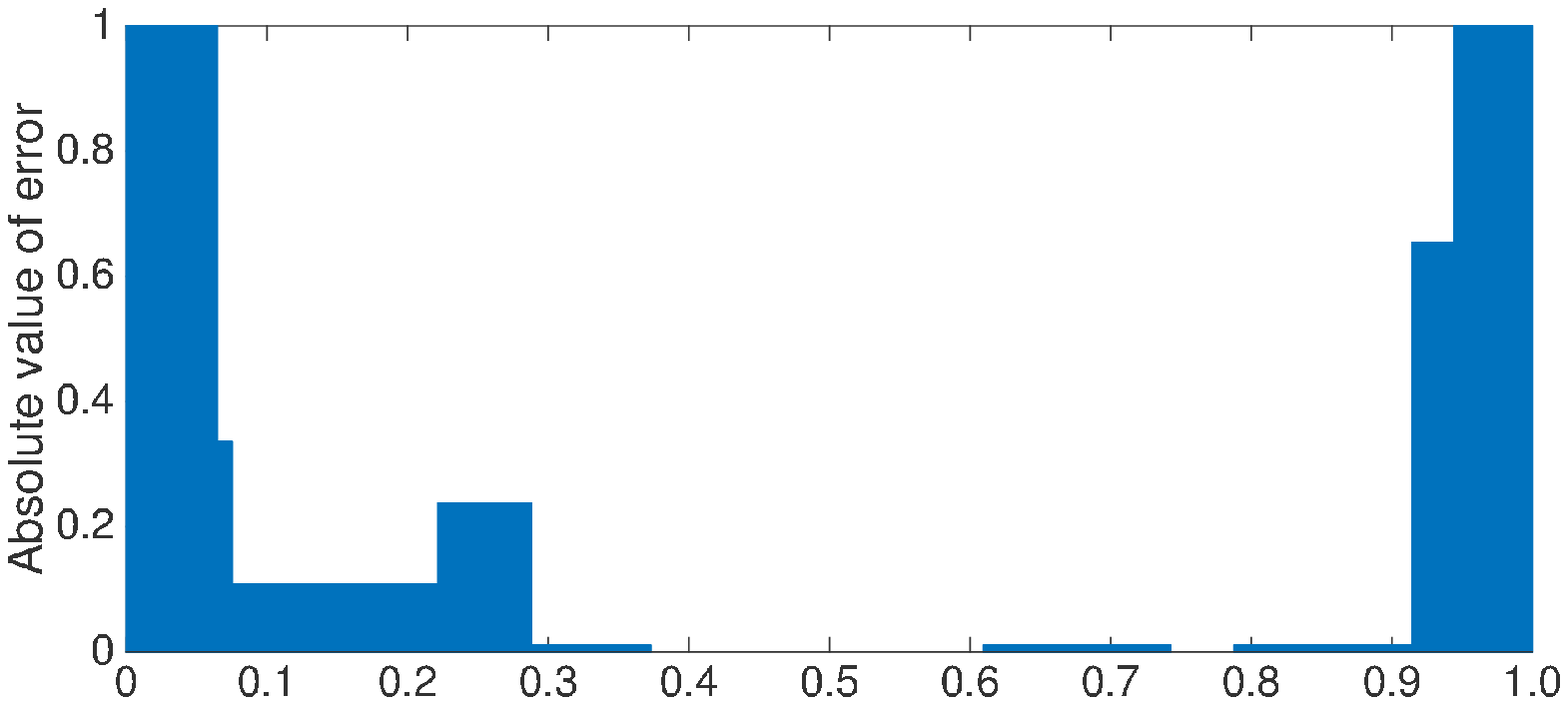}

\small (a)

\includegraphics[width=300pt]{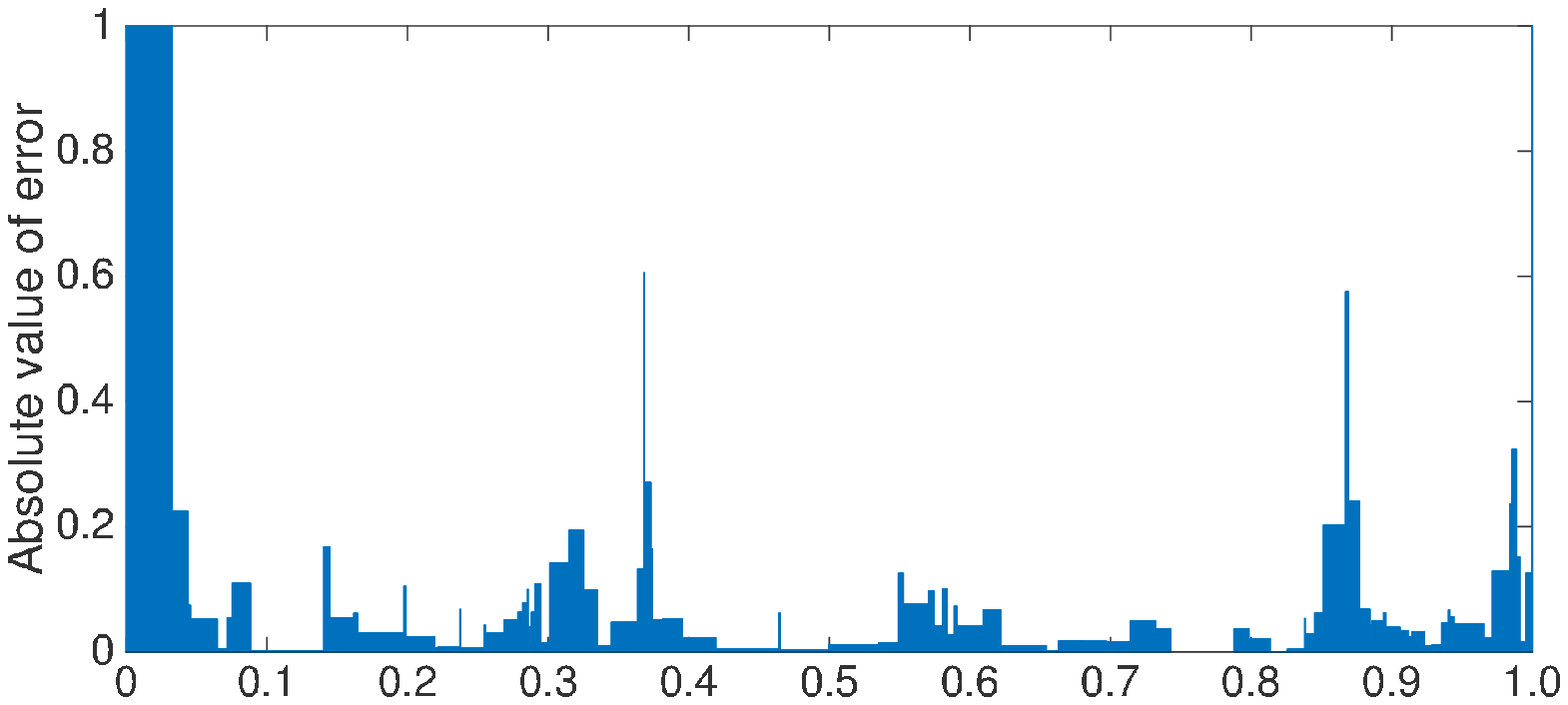}

\small (b)

\includegraphics[width=300pt]{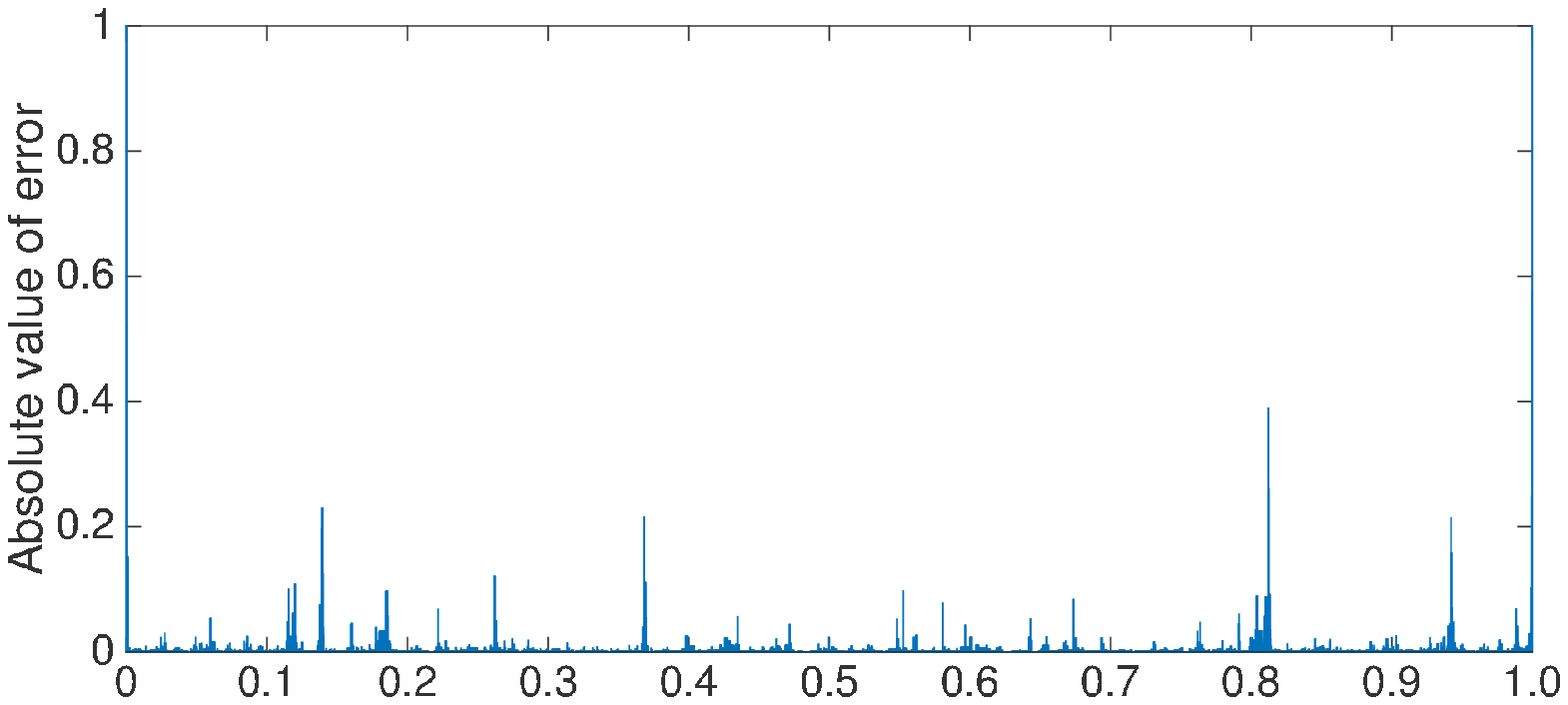}

\small (c)
\end{center}
\caption{Error function $f^\ast(x)-f_N(x)$ for $N=5$ (panel (a)), $N=50$ (panel (b)), and $N=500$ (panel (c)) for a typical curve from Fig. \ref{fig:ensemble_errors}.}\label{fig:errors_approx_individual}
\end{figure}
As we can see from the these figures the error functions become more and more pitchy or spiky with $N$. The individual spikes are randomly distributed on $[0,1]$, and their thickness is converging to zero.  To compensate for such errors one needs to be able to generate a very narrow $\varphi(a_i,\sigma_i,\cdot)$ which, in addition, is to be placed in the right location. Effectively, if error reduction at each step is sought for, this is equivalent to dimensionality growth of the problem at each step. However, in accordance with (\ref{eq:exponential_dimension}) to overcome negative effect of measure concentration, one needs to accumulate exponentially large number of samples. This is reflected in very slow convergence at the end of the process.


\section{Conclusion}\label{sec:Conclusion}

In this work we demonstrate that, despite increasing popularity of
random basis function networks in control literature, especially in
the domain of intelligent/adaptive control, one needs to pay special
attention to practical aspects that may affect performance of these
systems in applications.

First, as we analyzed in Section \ref{sec:Analysis} and showed in our examples,
although the rate of convergence of the random basis function
approximator is qualitatively similar to that of the greedy
approximator, the rate of the random basis function approximator is
{\it statistical}. In other words, small approximation errors are guaranteed here in
{\it probability}. This means that, in some applications such as e.g. practical adaptive control
in which the RVFL networks are to model or compensate system
uncertainties, employment of a re-initialization with a supervisory
mechanism monitoring quality of the RVFL network is necessary.
Unlike network training methods that adjust both linear and
nonlinear weights of the network, such mechanism may have to be made
robust against numerical problems (ill-conditioning).

Our conclusion about the random basis function approximators is also
consistent with the following intuition. If the approximating
elements (network nodes) are chosen at random and not subsequently
trained, they are usually not placed in accordance with the density
of the input data. Though computationally easier than for nonlinear
parameters, training of linear parameters becomes ineffective at
reducing errors ``inherited" from the nonlinear part of the
approximator. Thus, in order to improve effectiveness of the random
basis function approximators one could combine unsupervised
placement of network nodes according to the input data density with
subsequent supervised or reinforcement learning values of the linear
parameters of the approximator.  However, such a combination of
methods is not-trivial because in adaptive control and modeling one
often has to be able to allocate approximation resources adaptively
-- and the full network training seems to be the natural way to
handle such adaptation.

Second, we showed that in high dimensions exponentially large number of randomly and independently chosen vectors are almost orthogonal with probability close to one. This implies that in order to represent an element of such high-dimensional space by linear combinations of randomly and independently chosen vectors, it may often be necessary to generate samples of exponentially large length if we use bounded coefficients in linear combinations.  On the other hand, if coefficients with arbitrarily large values are allowed, the  number of randomly generated elements that are sufficient for approximation is even less than dimension. In the latter case, however, we have to pay for such a significant reduction of the number of elements by ill-conditioning of the approximation problem.

A simple numerical example that illustrates such behavior has been provided. Not only this example demonstrates the effect of measure concentration in a simple approximation problem, it also highlights practical implications for other approximation schemes that are based on randomization.

%

\bibliographystyle{plain}
\bibliography{randomized_literature}

\begin{thebibliography}{10}

\bibitem{INS:Wang:2014}
M.~Alhamdoosh and D.~H. Wang.
\newblock Fast decorrelated neural network ensembles with random weights.
\newblock {\em Information Sciences}, 264:104--117, 2014.

\bibitem{Artstein:2002}
S.~Artstein.
\newblock Proportional concentration phenomena of the sphere.
\newblock {\em Israel Journal of Mathematics}, 132:337--358, 2002.

\bibitem{Ball}
K.~Ball.
\newblock {\em An Elementary Introduction to Modern Convex Geometry},
  volume~31.
\newblock Flavors of Geometry, MSRI Publications, 1997.

\bibitem{Barron}
A.~R. Barron.
\newblock Universal approximation bounds for superposition of a sigmoidal
  function.
\newblock {\em IEEE Trans. on Information Theory}, 39(3):930--945, 1993.

\bibitem{Boyd94}
W.G.C. Boyd.
\newblock Gamma function asymptotics by an extension of the method of steepest
  descents.
\newblock {\em Proc. R. Soc. Lond. A}, 447(1931):609--630, 1994.

\bibitem{Tao}
E.J. Candes and T.~Tao.
\newblock Near-optimal signal recovery from random projections: Universal
  encoding strategies?
\newblock {\em IEEE Transactions on Information Theory}, 52(12):5406--5425,
  2006.

\bibitem{Cybenko}
G.~Cybenko.
\newblock Approximation by superpositions of a sigmoidal function.
\newblock {\em Math. of Control, Signals and Systems}, 2:303--314, 1989.

\bibitem{HOG}
N.~Dalal and B.~Triggs.
\newblock Histograms of oriented gradients for human detection.
\newblock In {\em Proceedings of the IEEE Computer Vision and Pattern
  Recognition Conference}, pages 886 -- 893. 2005.

\bibitem{Feldkamp98b}
L.A. Feldkamp, D.V. Prokhorov, C.F. Eagen, and F.~Yuan.
\newblock Enhanced multi-stream kalman filter training for recurrent networks.
\newblock In J.~Suykens and J.~Vandewalle, editors, {\em Nonlinear Modeling:
  Advanced Black-Box Techniques}, pages 29--53. Kluwer Academic Publishers,
  1998.

\bibitem{Gorban:1998}
A.N. Gorban.
\newblock Approximation of continuous functions of several variables by an
  arbitrary nonlinear continuous function of one variable, linear functions,
  and their superpositions.
\newblock {\em Appl. Math. Lett.}, 11(3):45--49, 1998.

\bibitem{Gorban:2006}
A.N. Gorban.
\newblock Order-disorder separation: Geometric revision.
\newblock {\em Physica A}, 374:85--102, 2007.

\bibitem{Gromov:1999}
M.~Gromov.
\newblock {\em Metric Structures for Riemannian and non-Riemannian Spaces. With
  appendices by M. Katz, P. Pansu, S. Semmes. Transpated from the French by
  Sean Muchael Bates}.
\newblock Birkhauser, Boston, MA, 1999.

\bibitem{GAFA:Gromov:2003}
M.~Gromov.
\newblock Isoperimetry of waists and concentration of maps.
\newblock {\em GAFA, Geomteric and Functional Analysis}, 13:178--215, 2003.

\bibitem{Haykin99}
S.~Haykin.
\newblock {\em Neural Networks: A Comprehensive Foundation}.
\newblock Prentice Hall, 1999.

\bibitem{he05}
P.~He and S.~Jagannathan.
\newblock Reinforcement learning based output feedback control of nonlinear
  systems with input contraints.
\newblock {\em {IEEE} {T}rans. {S}ystems, {M}an and {C}ybernetics},
  51(1):150--154, 2005.

\bibitem{Hecht}
R.~Hecht-Nielsen.
\newblock Context vectors: General-purpose approximate meaning representations
  self-organized from raw data.
\newblock In J.~Zurada, R.~Marks, and C.~Robinson, editors, {\em Computational
  Intelligence: Imitating Life}, pages 43 -- 56. IEEE Press, 1994.

\bibitem{Hornik90}
K.~Hornik, K.~Stinchcombe, and H.~White.
\newblock Universal approximation of an unknown mapping and its derivatives
  using multilayer neural networks.
\newblock {\em Neural Networks}, 3:551--560, 1990.

\bibitem{igelpao95}
B.~Igelnik and Y.-H. Pao.
\newblock Stochastic choice of basis functions in adaptive function
  approximation and the functional-link net.
\newblock {\em {IEEE} {T}rans. {N}eural {N}etworks}, 6(6):1320--1329, 1995.

\bibitem{Jones:1992}
L.~K. Jones.
\newblock A simple lemma on greedy approximation in {H}ilbert sapce and
  convergence rates for projection pursuit regression and neural network
  training.
\newblock {\em The Annals of Statistics}, 20(1):608--613, 1992.

\bibitem{Kurkova}
P.C. Kainen and V.~Kurkova.
\newblock Quasiorthogonal dimension of euclidian spaces.
\newblock {\em Appl. Math. Lett.}, 6(3):7--10, 1993.

\bibitem{LewisCampos02}
F.~Lewis, J.~Campos, and R.~Selmic.
\newblock {\em Neuro-fuzzy control of industrial systems with actuator
  nonlinearities}.
\newblock SIAM, 2002.

\bibitem{LewisGe06}
F.~Lewis and S.~Ge.
\newblock Neural networks in feedback control systems.
\newblock In M.~Kutz, editor, {\em Mechanical Engineers' Handbook}, volume~2,
  pages 791 -- 825. Wiley, third edition, 2006.

\bibitem{lichow97}
J.Y. Lin and W.S. Chow.
\newblock Comments on ``stochastic choice of basis functions in adaptive
  function approximation and the functional-link net''.
\newblock {\em {IEEE} {T}rans. {N}eural {N}etworks}, 8(2):452--454, 1997.

\bibitem{Liuetal07}
W.~Liu, J.~Sarangapani, G.~Venayagamoorthy, L.~Liu, D.~Wunsch, M.~Crow, and
  D.~Cartes.
\newblock Decentralized neural network-based excitation control of large-scale
  power systems.
\newblock {\em International Journal of Control, Automation, and Systems},
  5(5):526--538, 2007.

\bibitem{Annaswamy99}
Ai-Poh Loh, A.M. Annaswamy, and F.P. Skantze.
\newblock Adaptation in the presence of general nonlinear parameterization: An
  error model approach.
\newblock {\em IEEE Trans. on Automatic Control}, 44(9):1634--1652, 1999.

\bibitem{Pao:1994}
Y.H. Pao, G.H. Park, and D.J. Sobajic.
\newblock Learning and generalization characteristics of the random vector
  functional-link net.
\newblock {\em Neurocomputing}, 6(2):163--180, 1994.

\bibitem{Park93}
J.~Park and I.~W. Sandberg.
\newblock Approximation and radial basis function networks.
\newblock {\em Neural Computation}, 5(2):305--316, 1993.

\bibitem{Pestov}
V.~Pestov.
\newblock Is the k-nn classifier in high dimensions affected by the curse of
  dimensionality?
\newblock {\em Computers $\&$ Mathematics with Applications}, 65:1427--1437,
  2013.

\bibitem{tpt2002_at}
D.V. Prokhorov, V.A. Terekhov, and I.Yu. Tyukin.
\newblock On the applicability conditions for the algorithms of adaptive
  control in nonconvex problems.
\newblock {\em Automation and Remote Control}, 63(2):262--279, 2002.

\bibitem{Rahimi08}
A.~Rahimi and B.~Recht.
\newblock Uniform approximation of functions with random bases.
\newblock In {\em Proceedings of the 46th IEEE Annual Allerton Conference on
  Communication, Control and Computing}, pages 555 -- 561. 2008.

\bibitem{Rahimi08:2}
A.~Rahimi and B.~Recht.
\newblock Weighted sums of random kitchen sinks: Replacing minimization with
  randomization in learning.
\newblock In D.~Koller, D.~Schuurmans, Y.~Bengio, and L.~Bottou, editors, {\em
  Advances in Neural Information Processing Systems (NIPS)}, volume~2, pages
  1316--1323. 2008.

\bibitem{Resnikoff:2002}
H.~L. Resnikoff and R.~O. Wells.
\newblock {\em Wavelet Analysis}.
\newblock Springer, 2002.

\bibitem{Rosenblatt}
F.~Rosenblatt.
\newblock {\em Principles of Neurodynamics: Perceptrons and the Theory of Brain
  Mechanisms}.
\newblock Spartan Books, 1962.

\bibitem{Salhgren2005}
M.~Sahlgren.
\newblock An introduction to random indexing.
\newblock In H.~Witschel, editor, {\em Methods and Applications of Semantic
  Indexing Workshop at the 7th International Conference on Terminology and
  Knowledge Engineering. Vol. 87 of TermNet News: Newsletter of International
  Cooperation in Terminology}. 2005.
\newblock http://www.sics.se/∼mange/papers/RI intro.pdf.

\bibitem{Jag06}
J.~Sarangapani.
\newblock {\em Neural Network Control of Nonlinear Discrete-Time Systems}.
\newblock CRC Press, 2006.

\bibitem{Schmidt:1992}
W.~Schmidt, M.~Kraaijveld, and R.~Duin.
\newblock Feedforward neural networks with random weights.
\newblock In {\em Proceedings of 11th IAPR International Conference on Pattern
  Recognition Methodology and Systems}, pages 1--4. 1992.

\bibitem{Trefethen:2013}
L.~N. Trefethen.
\newblock {\em Approximation Theory and Approximation Practice}.
\newblock SIAM, 2013.

\bibitem{tpvl2007_tac}
I.Yu. Tyukin, D.~Prokhorov, and C.~van Leeuwen.
\newblock Adaptation and parameter estimation in systems with unstable target
  dynamics and nonlinear parametrization.
\newblock {\em {I}{E}{E}{E} {T}rans. on {A}utomatic {C}ontrol},
  52(9):1543--1559, 2007.

\bibitem{Tyukin:2009}
I.Yu. Tyukin and D.V. Prokhorov.
\newblock Feasibility of random basis function approximators for modeling and
  control.
\newblock In {\em Proceedings of the IEEE International Symposium on
  Intelligent Control ISIC'2009}, pages 1391 -- 1396. 2009.

\bibitem{tpt2003_tac}
I.Yu. Tyukin, D.V. Prokhorov, and V.A. Terekhov.
\newblock Adaptive control with nonconvex parameterization.
\newblock {\em {I}{E}{E}{E} {T}rans. on {A}utomatic {C}ontrol}, 48(4):554--567,
  2003.

\bibitem{Neural_Networks:Widrow:2013}
B.~Widrow, A.~Greenblatt, Y.~Kim, and D.~Park.
\newblock The no-prop algorithm: A new learning algorithm for multilayer neural
  networks.
\newblock {\em Neural Networks}, 37:182--188, 2013.

\end{thebibliography}

\end{document}